\documentclass[11pt,reqno]{amsproc}

\title[Singularity formation in Hele-Shaw]{On singularity formation in a Hele-Shaw model}

\author[P.~Constantin]{Peter Constantin}
\address{Department of Mathematics, Princeton University, Princeton, NJ 08544}
\email{const@math.princeton.edu}
\author[T.~Elgindi]{Tarek Elgindi}
\address{Department of Mathematics, UC San Diego, San Diego, CA 92093}
\email{telgindi@ucsd.edu}
\author[H.~Nguyen]{Huy Nguyen}
\address{Program in Applied and Computational Mathematics, Princeton University, Princeton, NJ 08544}
\email{qn@math.princeton.edu}
\author[V.~Vicol]{Vlad Vicol}
\address{Department of Mathematics, Princeton University, Princeton, NJ 08544}
\email{vvicol@math.princeton.edu}

\usepackage[margin=1in]{geometry}
\usepackage{amsmath, amsthm, amssymb}
\usepackage{mathrsfs}
\usepackage{graphicx}

\usepackage{times}
\usepackage[dvipsnames]{color}

\usepackage[colorlinks=true, pdfstartview=FitV, linkcolor=blue, citecolor=black, urlcolor=blue]{hyperref}

\newcommand{\huy}[1]{\textcolor{green}{#1}}

\newcommand{\bq}{\begin{equation}}
\newcommand{\eq}{\end{equation}}
\newcommand{\bqa}{\begin{eqnarray*}}
\newcommand{\eqa}{\end{eqnarray*}}
\newcommand{\Rr}{\mathbb{R}}

\newcommand{\la}{\label}

\newcommand{\be}{\begin{equation}}
\newcommand{\ee}{\end{equation}}
\newcommand{\ba}{\begin{array}{l}}
\newcommand{\ea}{\end{array}}

\theoremstyle{plain}
\newtheorem{theo}{Theorem}[section]
\newtheorem{prop}[theo]{Proposition}
\newtheorem{lemm}[theo]{Lemma}

\newtheorem{defi}[theo]{Definition}
\theoremstyle{definition}
\newtheorem{rema}[theo]{Remark}

\DeclareSymbolFont{pletters}{OT1}{cmr}{m}{sl}
\DeclareMathSymbol{s}{\mathalpha}{pletters}{`s}

\def\defn{\mathrel{:=}}

\def\eps{\varepsilon}

\def\la{\left\lvert}

\def\le{\leq}

\def\mez{\frac{1}{2}}

\def\ra{\right\rvert}

\def\tdm{\frac{3}{2}}

\def\xN{\mathbf{N}}

\def\cF{\mathcal{F}}

\def\p{\partial}
\def\wc{\rightharpoonup}
\def\D{\mathscr{D}}

\numberwithin{equation}{section}

\begin{document}

\begin{abstract}
We discuss a lubrication approximation model of the interface between two immiscible fluids in a Hele-Shaw cell, derived in ~\cite{CDGKSZ93} and widely studied since. The model consists of a single one dimensional evolution equation for the thickness $2h = 2h(x,t)$ of a thin neck of fluid,
\[
\partial_t h + \partial_x( h \, \partial_x^3 h) = 0\, ,
\]
for $x\in (-1,1)$ and  $t\ge 0$. The boundary conditions fix the neck height and the pressure jump: 
\[
h(\pm 1,t) = 1, \qquad \partial_{x}^2 h(\pm 1,t) = P>0.
\]
We prove that starting from smooth and positive $h$,  as long as $h(x,t) >0$, for $x\in [-1,1], \; t\in [0,T]$, no singularity can arise in the solution up to time $T$.  As a consequence, we prove for any $P>2$ and any smooth and positive initial datum that the solution pinches off in either finite or infinite time, {\it i.e.}, $\inf_{[-1,1]\times[0,T_*)} h = 0$, for some $T_* \in (0,\infty]$. These facts have been long anticipated on the basis of numerical and theoretical studies.
\hfill \today
\end{abstract}


\noindent

\maketitle

\section{Introduction}

In the Hele-Shaw problem, two immiscible viscous fluids are placed in a narrow gap between two plates. Neglecting variations transversal to the plates, the problem is modeled by two dimensional incompressible and irrotational hydrodynamical equations. In the presence of surface tension, boundary conditions connect the mean curvature of the interface separating the two fluids to the pressure jump. The fluids form characteristic patterns ~\cite{SaffmanTaylor58}.  The zero surface tension limit has been associated in {the} physical literature to Laplacian growth~\cite{wieg}, integrable systems~\cite{MWWZ00}, and to diffusion-limited aggregation~\cite{wit,Vicsek84,Halsey00}. A dimension reduction, using lubrication approximation, leads to degenerate fourth order parabolic equations in one space dimension. The original derivations are related to wetting, thin films, and the triple junction between two fluids and a solid substrate (see \cite{DeGennes85, SmythHill88,OronDavisBankoff97, BEIMR09} and \cite{DussanDavis74,Greenspan78,Hocking81}). Some of the mathematical papers related to the spreading of thin films and bubbles are \cite{BernisFriedman90,BertozziPugh96,BertozziPugh98,GiacomelliOtto03,GiacomelliKnupferOtto08,BernoffWitelski02,Knupfer15,KnupferMasmoudi15,GnannIbrahimMasmoudi17}.

Our focus in this paper is on singularity formation. In this context, a one dimensional model for topology change in a Hele-Shaw set-up was discussed in \cite{CDGKSZ93}. The equation describes the evolution of the thickness $h$ of a thin neck of fluid. The paper~\cite{CDGKSZ93} derives the evolution equation of $h$ using lubrication approximation, describes its variational dissipative structure and its steady states, and discusses the possibility of reaching zero thickness in finite or infinite time. This singularity formation was investigated theoretically and numerically in quite a number of studies. In \cite{DGKZ93} a first numerical evidence of finite time pinch off was obtained. Systematic expansions and numerical results for a wider range of problems indicated finite time pinch off and velocity singularities in \cite{GoldsteinPesciShelley93}. A family of equations was considered in \cite{BBDK94}, numerical results supporting  selfsimilar behavior were obtained, and finite or infinite time pinch off was asserted. In  \cite{EggersDupont94} numerical studies and physical arguments compared lubrication approximation equations to careful experiments of drop formation (\cite{ChaudharyRedekopp89,ChaudharyMaxworthy80,PeregrineShokerSymon90}). In \cite{CBEN99} experiment and scaling near equal viscosities are accompanied by studies of the dependence of the breaking rate and shape of the drop on the viscosity ratio. A comprehensive survey of selfsimilar behaviors is given in \cite{EggersFontelos08}, including a discussion of the pinch off scenarios presented on the basis of numerical evidence in \cite{AlmgrenBertozziBrenner96}. 
 
In spite of the remarkable success of the dramatically reduced model obtained by lubrication approximation (see \eqref{dp0}--\eqref{bc} below)  to quantitatively describe experimental reality, as evidenced by numerical studies and theoretical investigations, the finite time pinch off has yet to be rigorously proved.  In this paper, we prove an old conjecture of one of us, recorded in \cite{EggersDupont94}, that as long as $h>0$ no singularity can arise from smooth and positive initial data  (see Theorem~\ref{Cauchy} below). We also prove that indeed, as suggested in \cite{CDGKSZ93} and in \cite{BBDK94}, global in time behavior leads to pinch off, just as finite time singularities do (see Theorem~\ref{coro:stability} below). To the best of our knowledge, this is the first rigorous proof for the emergence of a pinching singularity in the one dimensional Hele-Shaw model of~\cite{CDGKSZ93}.

The equation we study (\cite{CDGKSZ93})
 \bq\label{dp0}
 \partial_th(x, t)+\partial_x(h \, \partial_x^3h)(x, t)=0, \quad  (x, t)\in  (-1, 1) \times (0, \infty),
 \eq
is supplemented with boundary conditions
\bq\label{bc}
\begin{aligned}
h(\pm 1, t)=1,\quad t>0,\\
\partial_x^2h(\pm 1, t)=P,\quad t>0.
\end{aligned}
\eq
Here, $P>0$ is the pressure of the less viscous fluid and $h\ge 0$ is half of the width of the thin neck.
The equation has a steady solution $h_P$, given by \eqref{eq:hp:def} below, which is unique in a class of relatively smooth solutions (see Proposition \ref{unique:hP}). This steady solution has a neck singularity if $P>2$ (a segment where it is identically zero). The main result of the paper is to prove convergence to this solution in finite or infinite time. In order to do so we start by obtaining a strong enough local existence result. We exploit further the structure of the equation to pass to limit of infinite time, and prove that the limits have to be formed from pieces of parabolas and straight lines where they do not vanish. Then we prove that the only possible valid limit there is $h_P$.

We denote $I=(-1, 1)$ and for any $T\in (0, \infty]$, we define
\[
X(T)=\left\{f\in L^\infty\big([0, T]; H^3(I)\big): {\p_x^3f \in L^2\big([0, T]; H^2(I)\big)} \right\}
\]
 endowed with its natural norm. When $T$ is finite, by interpolation  $X(T)$ is equivalent to the space
 \[
L^\infty\big([0, T]; H^3(I)\big)\cap L^2\big([0, T]; H^5(I)\big).
\]
\begin{theo}[\bf Local existence of strong solutions and continuation criterion]
\label{Cauchy}
 Let $h_0\in H^3(I)$ satisfy the boundary conditions \eqref{bc} and {assume} $h_{0, m}:=\inf_I h_0>0$. There exists a positive finite time $T$, depending only on {$P$}, $\Vert h_0\Vert_{H^3(I)}$ and $h_{0, m}$, such that problem \eqref{dp0}-\eqref{bc} with initial data $h_0$ has a unique solution $h\in X(T)$ with $\inf_{I\times [0, T]}h>0$.
 
Moreover, there exists an increasing function $\cF: {\Rr^+\times \Rr^+}\to \Rr^+$ depending only on $P$ such that
\bq\label{Cauchy:apriori}
\Vert h\Vert_{X(T)}\le \cF\big(\frac{1}{\inf_{I\times [0, T]}h}, \Vert h_0\Vert_{H^3}\big)
\eq
Therefore, $h$ blows up at a finite time $T^*$ if and only if 
\bq\label{blowup}
\inf_{x\in I} h(t, x)\searrow 0~\text{as}~t\nearrow T^*.
\eq
Furthermore, if we denote 
\bq
D(h(t))=\int_I h \, |\p_x^3h|^2(x, t)dx
\eq
 then 
 \bq\label{D:L1}
 \int_0^T D(h(t)) dt \le C(\Vert h_0\Vert_{H^3(I)}+1)
 \eq
  for some $C>0$ depending only on $P$, and
\bq\label{identity:D}
D(h(t))=D(h(0))+\int_0^t\Big(\int_I \p_th|\p_x^3h|^2(x, s)dxds-2\int_I |\p_x\p_th|^2(x, s)dx\Big)ds
\eq
 for {\it a.e.} $t \in [0, T]$.
 \end{theo}
\begin{rema}
We observe that the right-hand side of \eqref{Cauchy:apriori} does not explicitly depend on $T$. This fact is used in the proof of Theorem \ref{coro:stability}.
\end{rema}
The problem \eqref{dp0}-\eqref{bc} has the energy 
\[
E(h(t))=\mez\int_I|\partial_x h(x, t)|^2dx+P\int_I h(x, t)dx
\]
which dissipates according to
\[
\frac{d}{dt} E(h(t))=-D(h(t))=-\int_I h(x, t)|\partial_x^3h(x, t)|^2dx\le 0
\]
(see the proof of \eqref{energy:l} below). 

Define the steady solution $h_P$  by
\begin{subequations}
\label{eq:hp:def}
\begin{align}
h_P(x)&=\frac{P}{2}(x^2-1)+1, && P \in (0,2], \\
h_P(x) &=
\begin{cases}
\frac{P}{2}(|x|-x_P)^2,&\quad x_P\le |x|\le 1,\\
0,&\quad |x|<x_P, 
\end{cases} && P > 2,
\end{align}
\end{subequations}
where $x_P=1-\sqrt{\frac{2}{P}}$ for $P>2$. The energy dissipation rate $D(h)$ vanishes for $h=h_P$. When $P\in (0, 2]$, $h_P$ is a smooth, nonnegative solution of \eqref{dp0}-\eqref{bc}. When $P>2$, $h_P\in W^{2, \infty}(I)$ and has a jump of its second derivative at $\pm x_P$. In the second case, $h_P$ is a weak solution in the sense of the following definition.
\begin{defi}[\bf Weak solution]
\label{weak}
We say that a nonnegative function $h\in L^2([0, T]; H^2(I))$ is a weak solution of \eqref{dp0}-\eqref{bc} on $[0, T]$  if  there exists $\delta>0$ such that  for {\it a.e.} $t\in [0, T]$, $h(t)\in C^2([-1, -1+\delta])\cap C^2([1-\delta, 1])$, $h(t)$ verifies the boundary conditions \eqref{bc}, and  
\bq\label{weakform}
\int_0^T \!\!\! \int_I h\partial_t\varphi dxdt-\int_0^T\!\!\!\int_I\big(h\partial_x^2h-\mez|\partial_xh|^2\big)\partial_x^2\varphi dxdt=0
\eq
for all $\varphi\in C^\infty_0(I\times (0, T))$.
\end{defi}
The preceding definition is based on the identity 
\bq\label{id:nonl}
\partial_x(h\partial_x^3h)=\partial_x^2\big(h\partial_x^2h-\mez|\partial_xh|^2\big).
\eq
\begin{rema}[\bf Global weak solutions]
We prove in  {Theorem~\ref{theo:weak}} of the appendix that for any nonnegative $H^1$ data  {that is smooth near $\pm 1$ and satisfies the boundary condition \eqref{bc}}, there exists a global weak solution to \eqref{dp0}-\eqref{bc}. Related results for different boundary conditions can be found in \cite{BernisFriedman90, BertozziPugh96, BertozziPugh98}.
\end{rema}
The next proposition implies that $h_P$ has the least energy among all weak solutions.
\begin{prop}[\bf Energy minimizer]\label{prop:leaste}
 For any nonegative function $h\in H^1(I)$ taking value $1$ at $\pm 1$ we have $E(h)\ge E(h_P)$. Moreover, $E(h)=E(h_P)$ if and only if $h=h_P$. 
\end{prop}
 {In order to prove the finite or infinite time pinch off, we show that a sequence of functions with bounded energy $E$ and vanishing energy dissipation rate $D$ converges weakly to the energy minimizer $h_p$.}
\begin{theo}[\bf Relaxation to energy minimizer]
\label{theo:relax}
 Let $(h_n)$ be sequence of nonnegative $H^3(I)$ functions satisfying \eqref{bc}. Assume that $(h_n)$ is uniformly bounded in $H^1(I)$ and $D(h_n)\to 0$. Then we have $h_n\rightharpoonup h_P$ in $H^1(I)$ and $h_n\to h_P$ in $H^3_{loc}(\{x: h_P(x)>0\})$. When $P\in (0, 2)$, $h_n\to h_P$ in $H^3(I)$.
\end{theo}
As a corollary of Theorems \ref{Cauchy} and \ref{theo:relax}  we have  {the main result of this paper}:
\begin{theo}[\bf Stability for $P<2$ and pinch off for $P>2$] \label{coro:stability} \mbox{ } \\
{\bf Part 1.} If $P\in (0, 2)$, then $h_P$ is asymptotically stable in $H^1(I)$. More precisely, there exist $\delta, c, C>0$  {depending only on $P$} such that the following holds. If $h_0\in H^3(I)$ satisfies $\inf_I h_0>0$ and  $\Vert h_0-h_P\Vert_{H^1}\le \delta$ then $h\in X(T)$ for any $T>0$, $\inf_{I\times \Rr^+}h\ge c$ and 
\[
\Vert h(t)-h_P\Vert_{H^1(I)}\le C\Vert h_0-h_P\Vert_{H^1(I)}\exp(- {c} t)\quad\forall t>0.
\]
Moreover,  $h(t)\to h_P$ in $H^3(I)$ as $t\to \infty$. 

{\bf Part 2.} If $P\ge 2$, then starting from any positive $h_0\in H^3(I)$, the solution $h$ of \eqref{dp0}-\eqref{bc}, constructed in Theorem \ref{Cauchy}, pinches off at either finite or infinite time. In the latter case, by Theorem \ref{theo:relax}, $h(t_n)\rightharpoonup h_P$ in $H^1(I)$ and $h(t_n)\to h_P$ in $H^3_{loc}(\{x:h_P(x)>0\})$ for some $t_n\to \infty$.
\end{theo}
\begin{rema}
 When $P>2$, if $h$ is global in $X$, the bound \eqref{Cauchy:apriori} blows up since $h$ is pinched at infinite time. In particular, the bound for $h$ in $L^\infty([0, T]; H^3(I))$ blows up as $T\to \infty$. Nevertheless, along an unbounded sequence of times, $h$ converges to $h_P$ in $H^3_{loc}(\{x:h_P(x)>0\})$.
\end{rema}
{\begin{rema}
Assume that $h$ is a positive smooth solution of \eqref{dp0}-\eqref{bc} on $[0, T^*)$, $T^*\in (0, \infty)$, and that $\min_{x\in I} h(x, T^*)=0$. Let $x_m(t)$ be a  position of the minimum of $h$ in $x$ at time $t$ and denote $h_m(t)=h(x_m(t), t)$. Since $(\p_xh)(x_m(t), t)=0$, it is easy to see that
\[
\frac{d}{dt}\ln h_m(t)=-(\p_x^4h)(x_m(t), t)\quad\forall t\in [0, T^*).
\]
This implies
\[
\int_0^{T^*}(\p_x^4h)(x_m(t), t)dt=\infty.
\]
We also remark that in the derivation of model \eqref{dp0} (see \cite{CDGKSZ93}), the speed of the flow is given by $v=\p_x^3h$, and hence 
\[
\int_0^{T^*}(\p_xv)(x_m(t), t)dt=\infty.
\]
This is one kind of singularity occurring when $h$ touches $0$ in finite time. 
\end{rema} }

Throughout this paper, $\cF(\cdot, ...,\cdot)$ denotes nonnegative functions which are increasing in each argument. $\cF$ may change from line to line unless it is enumerated.
\section{A linear problem}
 Let $T$ be a positive real number and let  $g$ be a positive function satisfying 
\bq\label{cd:g}
g\in L^\infty([0, T]; H^2(I)),\quad \p_tg\in L^1([0, T]; L^\infty(I)).
\eq
 We study in this section the linear problem
\bq\label{dp:lin}
\begin{cases}
  \partial_th(x, t)+\partial_x(g\partial_x^3h)(x, t)=0, &\quad(x, t)\in I\times (0, T),\\
 h(\pm 1, t)=1,\quad \partial_x^2h(\pm 1, t)=P,&\quad t>0,\\
h(x, t)=h_0(x), &\quad t=0.
 \end{cases}
 \eq
We prove the following well-posedness result.
\begin{theo}[\bf Strong solution for the linear problem] \label{Cauchy:lin}
 For every $h_0\in H^3(I)$ satisfying the boundary conditions \eqref{bc}, there exists a unique solution $h\in X(T)$ to problem \eqref{dp:lin}. Moreover, denoting 
\[
c_0=\inf_{(x, t)\in I\times [0, T]}g(x, t)>0,
\]
then $h$ obeys the bounds
\begin{align}
\label{boundl1}
&\Vert h\Vert_{X(T)}\le \cF\big(\frac{1}{c_0},  \Vert g\Vert_{L^\infty([0, T]; H^2(I))}, \Vert \p_tg\Vert_{L^1([0, T]; L^\infty(I))}, \Vert h_0\Vert_{H^3}\big),\\
\label{boundl2}
&\Vert h\Vert_{L^\infty([0, T]; H^1(I))}\le C(1+\Vert h_0\Vert_{H^1(I)}),\\
\label{boundD:lin}
& {\int_0^T \!\!\! \int_I} g\, |\p_x^3h|^2\, dxdt \le C\big(1+\Vert h_0\Vert^2_{H^1(I)}\big).
\end{align}
Here, $\cF$ and $C$ depend only on $P$. Furthermore, denoting $w=g\, \p_x^3h$ we have that
\bq\label{identity:w}
\int_I \frac{w^2(x, t)}{g(x, t)}dx=\int_I \frac{w^2(x, 0)}{g(x, 0)}dx+\int_0^t \!\!\! \int_I \frac{\p_tg}{g^2}w^2(x, s)dxds-2\int_0^t\!\!\! \int_I |\p_x^2w|^2(x, s)dxds,
\eq
\bq\label{energy:w1}
\left\Vert \frac{w}{\sqrt g}(\cdot, t)\right\Vert_{L^2(I)}\le \left\Vert \frac{w}{\sqrt g}(\cdot, 0)\right\Vert_{L^2(I)}+\mez\int_0^t \left\Vert \frac{\p_tg}{g^\tdm}(\cdot, s)\right\Vert_{L^2(I)}\Vert w(\cdot, s)\Vert_{L^\infty(I)}ds ,
\eq
and
\bq\label{energy:w2}
\int_0^t\Vert \p_x^2w(\cdot, s)\Vert_{L^2}^2ds\le \mez \left\Vert \frac{w}{\sqrt g}(\cdot, 0)\right\Vert_{L^2(I)}^2+\mez\int_0^t \left\Vert \frac{\p_tg}{g^\tdm}(\cdot, s)\right\Vert_{L^2(I)}\left\Vert \frac{w}{\sqrt g}(\cdot, s)\right\Vert_{L^2(I)}\Vert w(\cdot, s)\Vert_{L^\infty(I)}ds
\eq
 {hold for {\it a.e.} $t\in [0, T]$.}
\end{theo}
 {The remainder of this section contains the proof of Theorem~\ref{Cauchy:lin}}.
Let $(g^n)$ a sequence of $C^\infty([0, T]\times \overline I)$ functions such that $g^n (x,t) \geq c_0/2$ and 
\bq\label{converge:g}
g^n\to g\in L^\infty([0, T]; H^2(I)),\quad \p_tg^n\to \p_tg\in L^1([0, T];  L^\infty(I)).
\eq
Let $h_0^n$ be a sequence of $C^\infty(\overline I)$ functions satisfying \eqref{bc} and converging to $h_0$ in $H^3(I)$. By the classical parabolic theory (see Theorem 6.2 \cite{LionsMagenes72}), there exists for each $n$ a unique solution $h^n \in C^\infty(\overline I)$
to the problem \eqref{dp:lin} with $g$ replaced by $g^n$ and $h_0$ replaced by $h_0^n$. We prove a closed a priori estimate for $h^n$ in $X(T)$, a contraction estimate in $H^1(I)$,  {and} then pass to the limit $n\to \infty$ to obtain the existence and uniqueness of a $h\in X(T)$ solving \eqref{dp:lin}.  To this end, we set
\[
 {u^n=h^n-\frac{P}{2}(x^2-1) - 1.}
\] 
Then,
\bq\label{eq:u:l}
 \partial_tu^n=-\partial_x(g^n \, \partial_x^3u^n)\quad\text{on}~[0, T],\quad u^n\vert_{t=0}=  {h_0^n-\frac{P}{2}(x^2-1)-1},
\eq
and 
\[
u^n(\pm 1, \cdot)=  {0} ,\quad u^n_{xx}(\pm 1,\cdot)=0.
\]
Throughout sections \ref{section:H1}, \ref{section:H2} and \ref{section:H3:l} we write $u^n=u$, $h^n=h$, $h_0^n=h_0$ and $g^n=g$ to simplify notation.
\subsection{$H^1$ energy}\label{section:H1}
We first claim that $h$ satisfies
\bq\label{energy:l}
\frac{d}{dt}\int_I (\mez | \partial_xh|^2+P h)=-\int_I g| \partial_x^3h|^2\le 0.
\eq
Indeed, we have
\begin{align*}
\frac{d}{dt}\int_I \mez |\partial_xh|^2&=\int_I \partial_t \partial_xh\partial_xh=\partial_th\partial_xh\vert_{-1}^1-\int_I \partial_th \partial_x^2h\\
&=\int_I \partial_x(g\partial_x^3h)\partial_x^2h=-\int_I g|\partial_x^3h|^2+Pg\partial_x^3h\vert_{-1}^1,
\end{align*}
and
\[
\frac{d}{dt}\int_I Ph=-P\int_I \partial_x(g\partial_x^3h)=-Pg\partial_x^3h\vert_{-1}^1,
\]
where we use the fact that $\partial_th(\pm 1, \cdot)=0$ (because $h(\pm 1, \cdot)=1$). This proves \eqref{energy:l}.

Next, multiplying \eqref{eq:u:l} by $-\partial_x^2u$, then integrating by parts we get
\[
-\int_I  \partial_tu\partial_x^2u=\int_I \partial_x(g\partial_x^3u)\partial_x^2u=g\partial_x^3u\partial_x^2u\vert_{-1}^1-\int_I g|\partial_x^3u|=-\int_I g|\partial_x^3u|^2.
\]
But
\[
-\int_I  \partial_tu\partial_x^2u=- \partial_tu \partial_xu\vert_{-1}^1+\int_I \partial_t\partial_x u\partial_xu=\mez\frac{d}{dt}\int_I | \partial_xu|^2
\]
noticing that $\partial_tu(\pm 1, \cdot)=0$ (because  {$u(\pm 1, \cdot)=0$}). Denoting 
\[
E_1=\Vert  \partial_xu\Vert_{L^2(I)},\quad D_1=\Vert \sqrt g \partial_x^3u\Vert_{ {L^2(I)}},
\]
we obtain
\bq\label{dtE1:l}
\mez\frac{d}{dt} E_1^2+D_1^2=0,
\eq
and hence
\bq\label{E1:l}
\mez E_1(T)^2+\Vert D_1\Vert_{L^2([0, T])}^2=\mez E_1(0)^2.
\eq
In particular, \eqref{E1:l} and the definition of $u$ gives 
\bq\label{bound:ux:l}
\Vert  \partial_x u\Vert_{L^2(I)}\le \Vert  \partial_xu(0)\Vert_{L^2(I)}  {\le \Vert \partial_x h_0 \Vert_{L^2(I)} + P}.
\eq
Since $u(\pm 1,t) = 0$, the Poincar\'e inequality also gives
\begin{align}
\Vert u\Vert_{L^2(I)} \leq C \Vert  \partial_x u\Vert_{L^2(I)} \leq C (1 + \Vert h_0 \Vert_{H^1(I)})
\end{align}
which implies together with \eqref{bound:ux:l} and the definition of $u$ that
\bq\label{H1:l}
\Vert h\Vert_{L^\infty([0, T]; H^1(I))}\le C  {(1 + \Vert h_0\Vert_{H^1(I)})}.
\eq
where $C$ only depends on $P$.

Moreover, by \eqref{E1:l} we obtain
\bq\label{boundD:lin1}
 {\int_0^T \!\!\! \int_I} g \, |\p_x^3h|^2 dxdt \le C\big(1+\Vert h_0\Vert^2_{H^1(I)}\big),
\eq
and by the positivity of $g$, 
\bq\label{d3:L2}
\Vert \partial_x^3h\Vert_{L^2([0, T]; L^2(I))}\le \frac{C}{\sqrt{c_0}}\big(1+\Vert h_0\Vert_{H^1(I)}\big)
\eq
where $c_0$ is as in the statement of the theorem.

\subsection{$H^2$ energy}\label{section:H2}
We multiply \eqref{eq:u:l} by $\partial_x^4u$ and integrate. On one hand,
\begin{align*}
\int_I \partial_tu\partial_x^4u&=\partial_tu\partial_x^3u\vert_{-1}^1-\int_I \partial_t\partial_xu\partial_x^3u=-\int_I \partial_t\partial_xu\partial_x^3u\\
&=-\partial_t\partial_xu\partial_x^2u\vert_{-1}^1+\int_I \partial_t\partial_x^2u\partial_x^2u=\mez\frac{d}{dt}\int_I |\partial_x^2u|^2.
\end{align*}
On the other hand,
\[
-\int_I \partial_x(g\partial_x^3u)\partial_x^4u=-\int_I g| {\partial_x^4} u|^2-\int_I \partial_x g \partial_x^3u\partial_x^4u.
\]
Denoting
\[
E_2=\Vert \partial_x^2u\Vert_{L^2(I)},\quad D_2=\Vert \sqrt g \partial_x^4u\Vert_{L^2(I)},
\]
it follows that
\[
\mez\frac{d}{dt} E_2^2+D^2_2=-\int_I \partial_x g \partial_x^3u \partial_x^4u\le \frac{1}{c_0}\Vert  \partial_x g\Vert_{L^\infty(I\times [0, T])}D_1D_2\le   \frac{1}{2c_0^2}\Vert  \partial_x g\Vert^2_{L^\infty(I\times [0, T])}D_1^2+\mez D_2^2.
\]
In view of \eqref{E1:l}, this yields
\begin{align*}
E^2_2(T)+\int_0^T D_2^2 dt 
&\le E^2_2(0)+\frac{1}{c^2_0}\Vert  \partial_x g\Vert^2_{L^\infty(I\times [0, T])}\int_0^T D_1^2 dt \\
&\le E^2_2(0)+\frac{1}{ {2} c^2_0}\Vert  \partial_x g\Vert^2_{L^\infty(I\times [0, T])}E_1^2(0),
\end{align*}
and consequently, 
\bq\label{d24:l}
\Vert \partial_x^2h\Vert_{L^\infty([0, T]; L^2(I))}+  {\sqrt{c_0}}\Vert \partial_x^4h\Vert_{L^2([0, T]; L^2(I))}\le 
 {C(1+\Vert h_0\Vert_{H^2})} + \frac{1}{c_0}\Vert  \partial_x g\Vert_{L^\infty(I\times [0, T])}(\Vert h_0\Vert_{H^1(I)}+C).
\eq
This, together with \eqref{H1:l} implies
\bq\label{H2:l}
\Vert h\Vert_{L^\infty([0, T]; H^2(I))}+  {\sqrt{c_0}} \Vert \partial_x^4h\Vert_{L^2([0, T]; L^2(I))}\le   {C(1+\Vert h_0\Vert_{H^2})} + \frac{1}{c_0}\Vert  \partial_x g\Vert_{L^\infty(I\times [0, T])}(\Vert h_0\Vert_{H^1(I)}+C).
\eq
\subsection{$H^3$ energy}\label{section:H3:l}
A direct $L^2$ estimate for $\p_x^3u$ would make  high order boundary terms  {appear} (up to order $5$) which are not given by the boundary conditions.  {Instead}, we exploit further the structure of the equation. Setting $w=g\, \p_x^3h$,  {we have} $\p_th=-\p_xw$, and thus $\p_xw(\pm 1)=\p_x^3w(\pm 1)=0$ in view of \eqref{bc}.  {From the identity}
\[
\p_t w=\p_tg\p_x^3h+g\p_x^3\p_th=\frac{\p_tg}{g}w-g\p_x^4w
\]
 {we conclude}
\bq\label{energy:w}
\begin{aligned}
\mez\frac{d}{dt}\int_I \frac{w^2}{g}&=\int_I \p_t w\frac{w}{g}-\mez\int_I \frac{\p_t g}{g^2}w^2=\int_I \frac{\p_tg}{g^2}w^2-\int_I w\p_x^4 w-\mez\int_I \frac{\p_t g}{g^2}w^2\\
&=\mez\int_I \frac{\p_tg}{g^2}w^2-\int_I w\p_x^4 w.
\end{aligned}
\eq
Integrating by parts twice and using the boundary conditions for $w$ gives
\[
\int_I w\p_x^4 w=\int_I |\p_x^2w|^2,
\]
which yields
\bq\label{identity:wn}
\int_I \frac{w^2(x, t)}{g(x, t)}dx=\int_I \frac{w^2(x, 0)}{g(x, 0)}dx+\int_0^t \!\!\! \int_I \frac{\p_tg}{g^2}w^2(x, s)dxds-2\int_0^t \!\!\! \int_I |\p_x^2w|^2(x, s)dxds,
\eq
\bq\label{dtw:1}
\mez\frac{d}{dt}\Vert \frac{w}{\sqrt g}\Vert_{L^2}^2+\Vert\p_x^2w\Vert_{L^2}^2\le \mez\Vert \frac{\p_tg}{g^\tdm}\Vert_{L^2}\Vert \frac{w}{\sqrt g}\Vert_{L^2}\Vert w\Vert_{L^\infty}
\eq
and
\bq\label{dtw:2}
\mez\frac{d}{dt}\Vert \frac{w}{\sqrt g}\Vert_{L^2}^2+\Vert\p_x^2w\Vert_{L^2}^2\le \mez\Vert \frac{\p_tg}{g}\Vert_{L^\infty}\Vert \frac{w}{\sqrt g}\Vert_{L^2}^2.
\eq
By \eqref{dtw:2} and Gr\"onwall's lemma,
\bq\label{est:w:lin}
\Vert \frac{w}{\sqrt g}\Vert_{L^\infty([0, T]; L^2)}+\Vert\p_x^2w\Vert_{L^2([0, T]; L^2)}\le \Vert \frac{w_0}{\sqrt g_0}\Vert_{L^2}\exp\big(2\int_0^T\Vert \frac{\p_tg}{g}\Vert_{L^\infty} ds\big).
\eq
Moreover, since 
\[
\p_x^2w=\p_x^2g\p_x^3h+2\p_xg\p_x^4h+g\p_x^5h
\]
and 
\bq\label{d3d4}
\Vert \p_x^3h\Vert_{L^\infty(I)}\le C\Vert \p_x^4h\Vert_{L^2(I)},
\eq
which follows from Poincar\'e-Wirtinger's inequality and the fact that \[
\int_I\p_x^3hdx=\p_x^2h(1)-\p_x^2h(-1)=P-P=0,
\]
we get
\bq\label{w-H5}
\begin{aligned}
\Vert g\p_x^5h\Vert_{L^2}&\le \Vert \p_x^2g\Vert_{L^2}\Vert \p_x^3h\Vert_{L^\infty}+2\Vert \p_xg\Vert_{L^\infty}\Vert \p_x^4h\Vert_{L^2}+\Vert \p_x^2w\Vert_{L^2}\\
&\le C\Vert g\Vert_{H^2}\Vert \p_x^4h\Vert_{L^2}+\Vert \p_x^2w\Vert_{L^2}.
\end{aligned}
\eq
In view of \eqref{H2:l}, \eqref{est:w:lin},  {\eqref{w-H5}, and the lower bound $g \geq c_0$}, we thus obtain
\bq\label{H3:l}
\Vert \p_x^3h\Vert_{L^\infty([0, T]; L^2)}+\Vert \p_x^5h\Vert_{L^2([0, T]; L^2)}\le \cF\big(\frac{1}{c_0}, \Vert   g\Vert_{L^\infty([0, T]; H^2)}, \Vert \p_tg\Vert_{L^1([0, T]; L^\infty)}, \Vert h_0\Vert_{H^3}\big).
\eq
\subsection{Proof of Theorem \ref{Cauchy:lin}}
A combination of \eqref{H2:l}, \eqref{d3:L2}
 and \eqref{H3:l} leads to
\bq\label{apriori:l}
\begin{aligned}
\Vert h^n\Vert_{X(T)}&\le \cF\big(\frac{1}{c_0},  \Vert g^n\Vert_{L^\infty([0, T]; H^2(I))},\Vert \p_tg^n\Vert_{L^1([0, T]; L^\infty)}, \Vert h_0^n\Vert_{H^3}\big)\\
&\le \cF\big(\frac{1}{c_0},  \Vert g\Vert_{L^\infty([0, T]; H^2(I))}, \Vert \p_tg\Vert_{L^1([0, T]; L^\infty)}, \Vert h_0\Vert_{H^3}\big).
\end{aligned}
\eq
Recall that $\p_th^n=-\p_xw^n$ and $\p_xw^n(\pm 1)=0$. It then follows from Poincar\'e's inequality and \eqref{est:w:lin} that
 \bq\label{apriori:l2}
\Vert \partial_th^n\Vert_{L^2([0, T]; H^1)}\le C  {\Vert  w^n\Vert_{L^2([0, T]; H^2)}} \le \cF\big(\frac{1}{c_0}, \Vert \p_tg\Vert_{L^1([0, T]; L^\infty)}, \Vert h_0\Vert_{H^3}\big).
\eq
By virtue of Aubin-Lions's lemma applied with the triple $H^3(I)\subset C^2(\overline I)\subset H^1(I)$, there exists $h\in X(T)$ such that
\begin{align}
\label{conv1:l}
&h^n\rightharpoonup h\quad\text{in}~L^2([0, T]; H^5(I))),\\
&h^n\rightharpoonup *~h\quad\text{in}~L^\infty([0, T]; H^3(I))
\label{conv2:l},\\
&h^n\to h\quad\text{in}~C([0, T]; C^2(\overline I)). \label{conv3:l}
\end{align}
For any test function $\phi\in C^\infty_0(I\times (0, T))$, 
\[
\int_0^T\int_I h^n\partial_t \phi dxdt+\int_0^T\int_I g^n\partial_x^3h^n\partial_x\phi dxdt=0.
\]
The convergences \eqref{conv1:l} and \eqref{converge:g} ensure that  $(h, g)$ satisfies the same weak formulation.  Then because $h\in L^2([0, T]; H^4(I))$ and $g\in L^\infty([0, T]; H^2(I))$, we actually have $\partial_th+\partial_x(g\partial_x^3h)=0$ in $L^2([0, T]; H^1)$.  Next, \eqref{conv3:l} implies that $h(0)=h_0$ and the boundary conditions $\partial_xh(\pm 1, t)=1$, $\partial_x^2h(\pm 1, t)=P$ are observed for any $t\in [0, T]$. The bounds \eqref{boundl1}, \eqref{boundl2} and \eqref{boundD:lin} on $h$ are inherited from the corresponding bounds \eqref{apriori:l}, \eqref{H1:l} and \eqref{boundD:lin1} on $h^n$. Letting $n\to \infty$ in \eqref{identity:wn} yields \eqref{identity:w}. Finally, integrating \eqref{dtw:1} and letting $n\to \infty$ we obtain \eqref{energy:w1} and \eqref{energy:w2}.

 {The uniqueness of solutions follows from the energy inequality.} Let $h_1$, $h_2$ be two solutions of \eqref{dp:lin} with the same initial condition $h_0$. The difference $k=h_1-h_2$ solves
\bq\label{diff:l}
\begin{cases}
 \partial_tk(x, t)+\partial_x(g\partial_x^3k)(x, t)=0, &\quad(x, t)\in I\times (0, T),\\
 k(\pm 1, t)=\partial_x^2k(\pm 1, t)=0,&\quad t>0,\\
k(x, t)=0, &\quad t=0.
 \end{cases}
 \eq
 {Similarly to the $H^1$ energy estimate for $u$ above}, we multiply the first equation in \eqref{diff:l} by $-\partial_x^2k$ and integrate by parts to get
 \[
 \mez\frac{d}{dt}\Vert \partial_xk\Vert_{L^2(I)}^2=-\int_I g|\partial_x^3k|^2\le 0,
 \]
 consequently $\partial_xk=0$. Since $k(\pm 1)=0$ we conclude that $k=0$,  {concluding the proof of uniqueness.} 
\section{A nondegenrerate problem}
Fixing a small positive real number $\eps$, we prove in this section the global well-posedness of the following nondegenerate nonlinear parabolic problem
\bq\label{dp:nd}
\begin{cases}
 \partial_th(x, t)+\partial_x(\sqrt{ h^2+\eps^2}\partial_x^3h)(x, t)=0, &\quad(x, t)\in I\times (0, \infty),\\
h(\pm 1, t)=1, \partial_x^2h(\pm 1, t)=P, &\quad t>0,\\
h(x, t)=h_0(x), &\quad t=0.
 \end{cases}
 \eq
 \begin{theo}[\bf Strong solution for the nondegenerate nonlinear problem] \label{Cauchy:nd}
 For every $h_0\in H^3$ satisfying the boundary conditions \eqref{bc}, and for every $T>0$, there exists a unique solution $h\in X(T)$ to problem \eqref{dp:nd}. Moreover, $h$ obeys the bounds
\begin{align}
\label{bound:nd1}
&\Vert h\Vert_{X(T)}\le \cF\big(\frac{1}{\inf_{I\times [0, T]}|h|+\eps}, \Vert h_0\Vert_{H^3}\big),\\
\label{bound:nd2}
&\Vert h\Vert_{L^\infty([0, T]; H^1(I))}\le C (1+ \Vert h_0\Vert_{H^1(I)})
\end{align}
with $\cF$ and $C$ depending only on $P$. Furthermore, \eqref{boundD:lin}, \eqref{identity:w}, \eqref{energy:w1} and \eqref{energy:w2} hold with $g=\sqrt{h^2+\eps^2}$.
 \end{theo}
  \subsection{Uniqueness}\label{section:unique:nd}
 If $h_1$ and $h_2$ are two solutions of \eqref{dp:nd}, we set $k=h_1-h_2$ and $g_j=\sqrt{h_j^2+\eps^2}$, $j=1, 2$. Observe that $k$ solves
  \bq\label{diff:hn}
\begin{cases}
 \partial_tk(x, t)+\partial_x(g_1\partial_x^3k)(x, t)+\partial_x((g_1-g_2)\partial_x^3h_2)(x, t)=0, &\quad(x, t)\in I\times (0, \infty),\\
 k(\pm 1, t)=\partial_x^2k(\pm 1, t)=0,&\quad t>0,\\
k(x, t)=0, &\quad t=0.
 \end{cases}
 \eq
 Multiplying the first equation in \eqref{diff:hn} by $-\partial_x^2k$ and integrating by parts (note that $\p_th_j\in L^2([0, T]; H^1_0(I))$) we get
 \begin{align*}
 \mez\frac{d}{dt}\Vert \partial_xk\Vert_{L^2(I)}^2&=-\int_I g_1|\partial_x^3k|^2-\int_I (g_1-g_2)\partial_x^3h_2\partial_x^3k\\
 &\le -\eps\int_I |\partial_x^3k|^2-\int_I (g_1-g_2)\partial_x^3h_2\partial_x^3k.
 \end{align*}
  It is readily seen that 
 \[
 \la g_1(x)-g_2(x)\ra \le \Vert h_1-h_2\Vert_{L^\infty(I)} \le C\Vert k\Vert_{H^1(I)}
 \]
 which implies
 \[
 \la \int_I (g_1-g_2)\partial_x^3h_2\partial_x^3k\ra \le C\Vert k\Vert_{H^1(I)}\Vert \partial_x^3h_2\Vert_{L^2(I)}\Vert \partial_x^3k\Vert_{L^2(I)}.
 \]
 Since $k(\pm 1, \cdot)=0$, Poincar\'e's inequality gives $
 \Vert k\Vert_{H^1(I)}\le C\Vert \partial_xk\Vert_{L^2(I)}$. 
 This combined with a Young inequality leads to
 \[
  \mez\frac{d}{dt}\Vert \partial_xk\Vert_{L^2(I)}^2\le C_\eps\Vert \partial_xk\Vert^2_{L^2(I)}\Vert \partial_x^3h_2\Vert^2_{L^2(I)}.
 \]
 Because $\Vert \partial_x^3h_2\Vert_{L^2(I)}\in L^2([0, T])$ for any $T>0$ we conclude by Gr\"onwall's lemma that $\partial_xk=0$ and thus $k=0$.
 \subsection{Local existence}\label{section:local-nd}
The existence of a local-in-time solution is obtained by Picard's iterations. We set $h^0(x, t)=h_0(x)$ for all $t>0$ and define recursively $h^{n+1}$, $n\ge 0$, to be the solution of the problem
 \bq\label{recur}
\begin{cases}
 \partial_th^{n+1}(x, t)+\partial_x(g^{n+1}\partial_x^3h^{n+1})(x, t)=0, &\quad(x, t)\in I\times (0, \infty),\\
 g^{n+1}=\sqrt{ |h^n|^2+\eps^2},\\
 h(\pm 1, t)=1,\quad\partial_x^2h(\pm 1, t)=P,&\quad t>0,\\
h(x, t)=h_0(x), &\quad t=0.
 \end{cases}
 \eq
 Applying recursively Theorem \ref{Cauchy:lin} we find that $h^n\in X(T)$ for any $T>0$.  We now prove by induction that there exist $T_0, C_0>0$,
 \[
 T_0=T_0(\frac{1}{\eps}, \Vert h_0\Vert_{H^3}),\quad  C_0=C_0(\frac{1}{\eps}, \Vert h_0\Vert_{H^3}),
 \]
  such that for any $n\ge 0$,
 \bq\label{uni:hn}
 \Vert h^n\Vert_{X(T_0)}+\Vert \p_th^n\Vert_{L^1([0, T_0]; L^\infty)}\le C_0.
 \eq
In view of the identities
\bq\label{dg}
\partial_t g^{n+1}=\frac{h^n\partial_th^n}{\sqrt{|h^n|^2+\eps^2}}, \quad\partial_x g^{n+1}=\frac{h^n\partial_xh^n}{\sqrt{|h^n|^2+\eps^2}},\quad \partial_{x}^2 g^{n+1}=\frac{|\partial_xh^n|^2+h^n\partial_x^2h^n}{\sqrt{|h^n|^2+\eps^2}}-\frac{|h^n|^2|\partial_xh^n|^2}{(|h^n|^2+\eps^2)^\tdm},
\eq
 we find
 \bq\label{est:gt}
\Vert \p_tg^{n+1}\Vert_{L^1([0, T]; L^\infty)}\le \Vert \p_th^n\Vert_{L^1([0, T]; L^\infty)}
\eq
 and
\bq\label{est:g}
\Vert g^{n+1}\Vert_{L^\infty([0, T]; H^2)}\le \cF_1(\frac{1}{\eps}, \Vert h^n\Vert_{L^\infty([0, T]; H^2)})
\eq
 This together with \eqref{boundl1} yields
\bq\label{direct:recur}
\Vert h^{n+1}\Vert_{X(T)}\le \cF_2(T, \frac{1}{\eps}, \Vert h^n\Vert_{L^\infty([0, T]; H^2)}, \Vert \p_th^n\Vert_{L^	1([0, T]; L^\infty)}, \Vert h_0\Vert_{H^3}).
\eq
Thus
\bq\label{direct:recur1}
\Vert h^{n+1}\Vert_{L^2([0, T]; H^5)}\le \cF_2(T, \frac{1}{\eps}, \Vert h^n\Vert_{L^\infty([0, T]; H^2)}, \Vert \p_th^n\Vert_{L^	1([0, T]; L^\infty)}, \Vert h_0\Vert_{H^3})
\eq
possibly with another $\cF_2$. From the equation for $h^{n+1}$ we deduce that 
\bq\label{direct:recur2}
\begin{aligned}
\Vert \p_th^{n+1}\Vert_{L^1([0, T]; L^\infty)}&\le \sqrt{T}\Vert \p_th^{n+1}\Vert_{L^2([0, T]; L^\infty)}\\
&\le C\sqrt{T}\Vert g^{n+1}\Vert_{L^\infty([0, T]; H^2)}\Vert h^{n+1}\Vert_{L^2([0, T]; H^5)}\\
&\le \sqrt{T}\cF_3(T, \frac{1}{\eps}, \Vert h^n\Vert_{L^\infty([0, T]; H^2)}, \Vert \p_th^n\Vert_{L^1([0, T]; L^\infty)}, \Vert h_0\Vert_{H^3}).
\end{aligned}
\eq
Thus \eqref{uni:hn}  holds for $n=0, 1$ with arbitrary $T_0\in (0, 1)$ and 
\bq\label{direct:01}
C_0>\max\Big\{\Vert h_0\Vert_{H^3}, \cF_2(1, \frac{1}{\eps}, \Vert h_0\Vert_{H^3}, 0, \Vert h_0\Vert_{H^3}), \cF_3(1, \frac{1}{\eps}, \Vert h_0\Vert_{H^3}, 0, \Vert h_0\Vert_{H^3})\Big\}=:M.
\eq
 Assume \eqref{uni:hn} for $0, 1,...n$ with $n\ge 1$ we now prove it for $n+1$. A direct induction based on \eqref{direct:recur} would amplify the bound for $h^{n+1}$, and thus additional considerations are needed. 
 \begin{lemm}
 There exist $\delta\in (0, 1)$ and $\cF_5$, $\cF_6$ such that for all $T\le 1$ and $n\ge 1$,
 \bq\label{recur:T}
 \begin{aligned}
 \Vert h^{n+1}\Vert_{X(T)}\le  \cF_5\Big(\frac{1}{\eps}, T^\delta \cF_6\big(\Vert h^{n-1}\Vert_{L^\infty([0, T]; H^2(I))},  \Vert \p_th^{n-1}\Vert_{L^2([0, T]; H^1)}\big), \Vert h_0\Vert_{H^3(I)} \Big).
 \end{aligned}
 \eq
 \end{lemm}
 \begin{proof}
 We first note that $u^{n}:=h^{n}-\frac P2 x^2$ solves 
  \bq\label{recur:u}
\begin{cases}
 \partial_tu^{n}(x, t)+\partial_x(g^{n}\partial_x^3u^{n})(x, t)=0, &\quad(x, t)\in I\times (0, \infty),\\
 g^{n}=\sqrt{ |h^{n-1}|^2+\eps^2},\\
 u^n(\pm 1, t)=\partial_x^2u^n(\pm 1, t)=0,&\quad t>0,\\
u^n(x, t)=u^n_0(x):=h_0(x)-\frac P2x^2, &\quad t=0.
 \end{cases}
 \eq
 Then as in section \ref{section:H2}, we multiply the first equation in \eqref{recur:u} by $\partial_x^4u^{n}$ and integrate by parts to obtain
 \[
 \mez\frac{d}{dt}\Vert \partial_x^2u^{n}\Vert^2_{L^2(I)}=-\int_I g^{n}|\partial_x^4u^{n}|^2-\int_I \partial_xg^{n}\partial_x^3u^{n}\partial_x^4u^{n}.
 \]
 Let us note that $\p_t\p_x^2u\in L^2([0, T]; H^{-1}(I))$ and $\p_x^2u\in L^2([0, T]; H_0^1(I))$. Employing the Gagliardo-Nirenberg inequality
 \[
 \Vert \partial_x^3 f\Vert_{L^2(I)}\le  C\Vert \partial_x^4 f\Vert^\alpha_{L^2(I)} \Vert f\Vert^{1-\alpha}_{L^2(I)}+ C\Vert f\Vert_{L^2(I)},\quad \alpha=\frac{3}{4},
 \]
 we bound
 \begin{align*}
 \la \int_I \partial_xg^{n}\partial_x^3u^{n}\partial_x^4u^{n}\ra&\le \Vert \partial_xg^{n}\Vert_{L^\infty(I)} \Vert \partial_x^3 u^{n}\Vert_{L^2(I)}\Vert \partial_x^4u^{n}\Vert_{L^2(I)}\\
& \le C\Vert g^{n}\Vert_{H^2(I)} \Vert \partial_x^4u^{n}\Vert^{1+\alpha}_{L^2(I)}\Vert u^{n}\Vert_{L^2(I)}^{1-\alpha}\\
&\quad + C\Vert g^{n}\Vert_{H^2(I)}\Vert \partial_x^4u^{n}\Vert_{L^2(I)}\Vert u^{n}\Vert_{L^2(I)}.
 \end{align*}
 Consequently
 \begin{align*}
 &\Vert \partial_x^2u^{n}\Vert^2_{L^\infty([0, T]; L^2(I))}+\Vert \sqrt{g^{n}}\partial_x^4u^{n}\Vert^2_{L^2([0, T]; L^2(I))}\\
 &\le  \Vert \partial_x^2u^{n}(0)\Vert^2_{L^2(I)}+C\Vert g^{n}\Vert_{L^\infty([0, T]; H^2(I))} \Vert u^{n}\Vert_{L^\infty([0, T]; L^2(I))}^{1-\alpha} \int_0^T\Vert \partial_x^4u^{n}\Vert^{1+\alpha}_{L^2(I)}\\
 &\quad+C\Vert g^{n}\Vert_{L^\infty([0, T]; H^2(I))}\Vert u^{n}\Vert_{L^\infty([0, T]; L^2(I))}\int_0^T\Vert \partial_x^4u^{n}\Vert_{L^2(I)}.
 \end{align*}
 Appealing to  H\"older's inequality we can gain small factors of powers of $T$:
 \[
\int_0^T\Vert \partial_x^4u^{n}\Vert^{1+\alpha}_{L^2(I)}\le T^{\frac{1-\alpha}{2}}\Vert \partial_x^4 u^{n}\Vert_{L^2([0, T]; L^2(I))}
^{\alpha+1},\quad  \int_0^T\Vert \partial_x^4u^{n}\Vert_{L^2(I)}\le T^\mez\Vert \partial_x^4 u^{n}\Vert_{L^2([0, T]; L^2(I))}.
 \]
Invoking \eqref{est:g} and \eqref{direct:recur} with $n$ replaced by $n-1$  leads to
 \begin{align*}
& \Vert \partial_x^2u^{n}\Vert_{L^\infty([0, T]; L^2(I))}+\Vert \sqrt{g^{n}}\partial_x^4u^{n}\Vert_{L^2([0, T]; L^2(I))}\\
&\le T^{\beta}\cF_3(\frac{1}{\eps}, \Vert h^{n-1}\Vert_{L^\infty([0, T]; H^2(I))}, \Vert \p_th^{n-1}\Vert_{L^1([0, T]; L^\infty)}, \Vert h_0\Vert_{H^3})+\Vert u_0\Vert_{H^2(I)}
 \end{align*}
 for some $\beta\in (0, 1)$ and for all $T\le 1$, $n\ge 1$.  We thus obtain by virtue of \eqref{boundl2},
   \[
  \Vert h^{n}\Vert_{L^\infty([0, T]; H^2)}\le T^{\beta}\cF_4\big(\frac{1}{\eps}, \Vert h^{n-1}\Vert_{L^\infty([0, T]; H^2)}, \Vert \p_th^{n-1}\Vert_{L^1([0, T]; L^\infty)}, \Vert h_0\Vert_{H^3}\big)+C\Vert h_0\Vert_{H^2}+C.
  \]
Substituting  this and \eqref{direct:recur2} (with $n$ replaced by $n-1$) in \eqref{direct:recur} yields
  \begin{align*}
 &\Vert h^{n+1}\Vert_{X(T)}\\
 &\le \cF_2\Big(T, \frac{1}{\eps}, T^{\beta}\cF_4\big(\frac{1}{\eps}, \Vert h^{n-1}\Vert_{L^\infty([0, T]; H^2)}, \Vert \p_th^{n-1}\Vert_{L^2([0, T]; H^1)},\Vert h_0\Vert_{H^3}\big)+C\Vert h_0\Vert_{H^2}+C,\\
 & \quad \sqrt{T}\cF_3\big(T, \frac{1}{\eps}, \Vert h^{n-1}\Vert_{L^\infty([0, T]; H^2)}, \Vert \p_th^{n-1}\Vert_{L^1([0, T]; L^\infty)}, \Vert h_0\Vert_{H^3}\big), \Vert h_0\Vert_{H^3}\Big)\\
 &\le \cF_5\Big(\frac{1}{\eps}, T^\gamma \cF_6\big(\Vert h^{n-1}\Vert_{L^\infty([0, T]; H^2)}, \Vert \p_th^{n-1}\Vert_{L^2([0, T]; H^1)}\big), \Vert h_0\Vert_{H^3} \Big)
 \end{align*}
 for some $\gamma\in (0, 1)$, for all $T\le 1$ and $n\ge 1$.   \end{proof}
 Now we choose 
 \[
 C_0>\max\Big\{M, \cF_5(\frac{1}{\eps}, 1, \Vert h_0\Vert_{H^3})\Big\}
 \]
 and $T_0\in (0, 1)$ satisfying
 \[
 T_0^\gamma\cF_6(C_0, C_0)\le 1,\quad \sqrt{T_0}\cF_3(\frac{1}{\eps}, C_0, C_0, \Vert h_0\Vert_{H^3})\le C_0
 \]
 then owing to \eqref{direct:recur2}, \eqref{recur:T} and the induction hypothesis,
 \[
 \Vert h^{n+1}\Vert_{X(T_0)}+ \Vert \p_th^{n+1}\Vert_{L^1([0, T_0]; L^\infty)}\le C_0
 \]
 which completes the proof of the uniform bounds \eqref{uni:hn}. In fact, using the first equation in \eqref{recur} and the uniniform boundedness of $h^n$ in $X(T_0)$ we deduce that $\partial_th^n$ is uniformly bounded in $L^2([0, T_0]; H^1(I))$.  Passing to the limit $n\to \infty$ with the use of  Aubin-Lions's lemma, we obtain a solution $h\in X(T_0)$ of \eqref{dp:nd}. Moreover, $T_0\in (0, 1)$ depends only on $\Vert h_0\Vert_{X}$ and $\eps$, and the bound 
  \[
 \Vert h\Vert_{X(T_0)}\le C_0\le \cF(\frac{1}{\eps}, \Vert h_0\Vert_{H^3})
 \]
 holds. Finally, \eqref{boundD:lin}, \eqref{identity:w}, \eqref{energy:w1} and \eqref{energy:w2} hold with $g=\sqrt{h^2+\eps^2}$ by applying Theorem \ref{Cauchy:lin} to \eqref{recur} then letting $n\to \infty$.
 \subsection{Global existence}\label{section:global-nd}
 We now iterate the above procedure over time intervals $\mathcal{T}_m$ of length less than $1$ and glue the solutions together to obtain a maximal solution $h$ defined on $[0, T^*)$ with $T^*\in (0, \infty]$.
 \begin{prop}
 For any $T<T^*$, $h$ obeys the bound
 \bq\label{apriori}
 \Vert h\Vert_{X(T)}\le \cF\big(\frac{1}{h_m(T)+\eps}, \Vert h_0\Vert_{H^3}\big),\quad h_m(T):=\inf_{I\times [0, T]}|h|.
 \eq
 \end{prop}
 \begin{proof}
We revisit the energy estimates leading to Theorem \ref{Cauchy:lin} but with $g$ replaced by $h$. First,  the inequality \eqref{energy:l} holds,
\[
\frac{d}{dt}\int_I (\mez | \partial_xh|^2+P h)=-\int_I g| \partial_x^3h|^2\le 0.
\]
Letting $ {u=h-\frac P2 (x^2-1) - 1}$ and $g=\sqrt{h^2+\eps^2}$, as in sections \ref{section:H1} and \ref{section:H2} we have that
\bq\label{apriori:E1}
\mez\frac{d}{dt}E_1^2+D_1^2\le 0
\eq
and
\[
\mez\frac{d}{dt} E_2^2+D^2_2=-\int_I \partial_x g \partial_x^3u \partial_x^4u
\]
hold, where
\[
E_1=\Vert  \partial_xu\Vert_{L^2(I)},\quad D_1=\Vert \sqrt g \partial_x^3u\Vert_{L^2},\quad E_2=\Vert  \partial_x^2u\Vert_{L^2(I)}, \quad D_2=\Vert \sqrt g \partial_x^4u\Vert_{L^2}.
\]
In particular, we deduce as for \eqref{H1:l} that
\bq\label{apriori:H1}
\Vert h\Vert_{L^\infty([0, T]; H^1(I))}\le  { C(1+ \Vert h_0\Vert_{H^1(I)})}.
\eq
Writing $\partial_x g=\p_xh \frac hg=(\partial_x u+Px)\frac hg$ and noting that $|h|\le g$ we bound
\begin{align*}
\la \int_I \partial_x g \partial_x^3u \partial_x^4udx\ra&\le\int_I \la\partial_x u \partial_x^3u \partial_x^4u\ra dx+P\int_I\la x \partial_x^3u \partial_x^4u\ra dx\\
&\le  \frac{1}{h_m(T)+\eps}\Vert \partial_xu\Vert_{L^\infty(I)}D_1D_2+\frac{P}{h_m(T)+\eps}D_1D_2\\
&\le  \frac{1}{h_m(T)+\eps}\Vert \partial_xu\Vert_{H^1(I)}D_1D_2+\frac{P}{h_m(T)+\eps}D_1D_2\\
&\le  \frac{C}{h_m(T)+\eps}E_2D_1D_2+\frac{P}{h_m(T)+\eps}D_1D_2, \\  
&\le \mez D_2^2+ \frac{C}{h^2_m(T)+\eps^2}E_2^2D_1^2+\frac{C}{h^2_m(T)+\eps^2}D_1^2
\end{align*}
where the bound 
\[
\Vert \partial_xu\Vert_{H^1(I)}\le C\Vert \partial^2_xu\Vert_{L^2(I)},
\]
which follows from Poincar\'e-Wirtinger's inequality together with the fact that $\int_I \p_xu=0$, was used. Thus
\[
\mez\frac{d}{dt} E_2^2 +\mez D_2^2\le \frac{C}{h^2_m(T)+\eps^2}E_2^2D_1^2+\frac{C}{h^2_m(T)+\eps^2}D_1^2
\]
which combined with \eqref{apriori:E1} yields
\[
\mez \frac{d}{dt}E^2+\mez D_2^2\le \frac{C}{h^2_m(T)+\eps^2}E_2^2D_1^2\le \frac{C}{h^2_m(T)+\eps^2}E^2D_1^2
\]
with $E^2=\frac{C}{h^2_m(T)+\eps^2}E_1^2+E_2^2$. Then by the Gr\"onwall lemma,
\bq\label{apriori:E2}
\begin{aligned}
\Vert E_2\Vert_{L^\infty([0, T])}\le \Vert E\Vert_{L^\infty([0, T])}&\le E(0)\exp(\frac{C}{h^2_m(T)+\eps^2}\Vert D_1\Vert^2_{L^2([0, T])})\\
&\le E(0)\exp(\frac{C}{h^2_m(T)+\eps^2}E^2_1(0)).
\end{aligned}
\eq
It follows that
\bq\label{apriori:D2}
\begin{aligned}
\Vert D_2\Vert_{L^2([0, T])}&\le \frac{C}{h_m(T)+\eps}\Vert E_2\Vert_{L^\infty([0, T])}\Vert D_1\Vert_{L^2([0, T])}\\
&\le  \frac{C}{h_m(T)+\eps}E(0)\exp(\frac{C}{h^2_m(T)+\eps^2}E^2_1(0))E_1(0).
\end{aligned}
\eq
A combination of \eqref{apriori:H1}, \eqref{apriori:E1}, \eqref{apriori:E2} and \eqref{apriori:D2} leads to
\bq\label{apriori:H24}
\Vert h\Vert_{L^\infty([0, T]; H^2(I))}+\Vert \p_x^3h\Vert_{L^2([0, T]; L^2(I))}+\Vert \p_x^4h\Vert_{L^2([0, T]; L^2(I))}\le \cF\big(\frac{1}{h_m(T)+\eps}, \Vert h_0\Vert_{H^2}\big). 
\eq
We now turn to the $H^3$ estimate. As proved in section \ref{section:local-nd}, \eqref{energy:w1} and \eqref{energy:w2} (with $g=\sqrt{h^3+\eps^2}$) hold on each iterative time interval $\mathcal{T}_m$, and thus hold on $[0, T]$ by gluing them together. In other words, we have for {\it a.e.} $t\in [0, T]$ that
\bq\label{energy:w1:nd}
\Vert \frac{w}{\sqrt g}(\cdot, t)\Vert_{L^2(I)}\le \Vert \frac{w}{\sqrt g}(\cdot, 0)\Vert_{L^2(I)}+\mez\int_0^t\Vert \frac{\p_tg}{g^\tdm}(\cdot, s)\Vert_{L^2(I)}\Vert w(\cdot, s)\Vert_{L^\infty} ds
\eq
and
\bq\label{energy:w2:nd}
\int_0^t\Vert \p_x^2w(\cdot, s)\Vert_{L^2}^2ds\le \mez\Vert \frac{w}{\sqrt g}(\cdot, 0)\Vert_{L^2(I)}^2+\mez\int_0^t\Vert \frac{\p_tg}{g^\tdm}(\cdot, s)\Vert_{L^2(I)}\Vert \frac{w}{\sqrt g}(\cdot, s)\Vert_{L^2(I)}\Vert w(\cdot, s)\Vert_{L^\infty}ds.
\eq
But by \eqref{d3d4} it is readily seen that
\[
\Vert w\Vert_{L^\infty}\le C\Vert g\Vert_{L^\infty}\Vert \p_x^3h\Vert_{L^\infty}\le C(\Vert h\Vert_{L^\infty}+\eps)\Vert \p_x^4h\Vert_{L^2}
\]
and
\bq\label{dt:g}
\Vert \p_tg\Vert_{L^2}\le \Vert \p_th\Vert_{L^2}\le C\Vert h\Vert_{H^2}(\Vert \p_x^3h\Vert_{L^2}+\Vert \p_x^4h\Vert_{L^2}).
\eq
 Consequently
\[
\Vert \frac{w}{\sqrt g}\Vert_{L^\infty([0, T]; L^2)}\le \Vert \frac{w_0}{\sqrt g_0}\Vert_{L^2}+\frac{C}{(h_m(T)+\eps)^\tdm}A
\]
with 
\begin{align*}
A&=\Vert w\Vert_{L^\infty([0, T]; L^\infty)}\Vert \p_tg\Vert_{L^2([0, T]; L^2)}\\\
&\le C(\Vert h\Vert_{H^1}+\eps)\Vert h\Vert_{L^\infty([0, T]; H^2)}\Big(\Vert \p_x^3h\Vert_{L^2([0, T]; L^2)}\Vert \p_x^4h\Vert_{L^2([0, T]; L^2)}+\Vert \p_x^4h\Vert^2_{L^2([0, T]; L^2)}\Big)\\
&\le \cF\big(\frac{1}{h_m(T)+\eps}, \Vert h_0\Vert_{H^2}\big)
\end{align*}
in view of \eqref{apriori:H24}, and
\begin{align*}
\Vert \p_x^2w\Vert_{L^2([0, T]; L^2)}^2&\le \mez\Vert \frac{w_0}{\sqrt g_0}\Vert_{L^2}^2+\frac{C}{(h_m(T)+\eps)^\tdm}\Vert \frac{w}{\sqrt g}\Vert_{L^\infty([0, T]; L^2)}\Vert \p_tg\Vert_{L^\infty([0, T]; L^2)}\Vert w\Vert_{L^\infty([0, T]; L^\infty)}\\
&\le \mez\Vert \frac{w_0}{\sqrt g_0}\Vert_{L^2}^2+\frac{C}{(h_m(T)+\eps)^\tdm}\Big( \Vert \frac{w_0}{\sqrt g_0}\Vert_{L^2}+\frac{C}{(h_m(T)+\eps)^\tdm}A\Big)A\\
&\le \cF\big(\frac{1}{h_m(T)+\eps}, \Vert h_0\Vert_{H^2}\big).
\end{align*}
Appealing to \eqref{w-H5} with $g=h$ we deduce that
\bq\label{apriori:H3}
\Vert \p_x^3h\Vert_{L^\infty([0, T]; L^2)}+\Vert \p_x^5h\Vert_{L^\infty([0, T]; L^2)}\le \cF\big(\frac{1}{h_m(T)+\eps}, \Vert h_0\Vert_{H^3}\big)
\eq
from which  \eqref{apriori} follows.
\end{proof}
 Now \eqref{apriori} implies the global bound
\[
 \Vert h\Vert_{X(T)}\le \cF\big(\frac{1}{\eps}, \Vert h_0\Vert_{H^3}\big)
\]
for any $T<T^*$. We thus conclude that $T^*=\infty$. Furthermore, the bounds \eqref{bound:nd1} and \eqref{bound:nd2} follow from   \eqref{apriori:H1}, \eqref{apriori:H3} and \eqref{apriori:H24}.
\section{Proof of Theorem \ref{Cauchy}}
Let $h_0\in H^3$ satisfy the boundary conditions \eqref{bc} and
\[
h_{0, m}:=\inf_I h_0>0.
\]
{\bf Step 1}. (Approximate equations). For each $\eps \in (0, 1]$, let $h_\eps$ be the solution of the nondegenrate problem
\bq\label{app:nd}
\begin{cases}
 \partial_th_\eps(x, t)+(\sqrt{h_\eps^2+\eps^2}\partial_x^3h_\eps)_x(x, t)=0, &\quad(x, t)\in (-1, 1)\times (0, \infty),\\
 h_\eps(\pm 1, t)=1,\quad\partial_x^2h_\eps(\pm 1, t)=P,&\quad t>0,\\
h_\eps(x, t)=h_0(x), &\quad t=0.
 \end{cases}
 \eq
  According to Theorem \ref{Cauchy:nd}, $h^\eps\in X(T)$ for any $T>0$ and $h_\eps$ obeys the bounds 
 \begin{align}
\label{bound:app1}
&\Vert h_\eps\Vert_{X(T)}\le \cF\big(\frac{1}{h_{\eps, m}(T)+\eps},  \Vert h_0\Vert_{H^3}\big),\\
\label{bound:app2}
&\Vert h_\eps\Vert_{L^\infty([0, T]; H^1(I))}\le C (1+ \Vert h_0\Vert_{H^1(I)})
\end{align}
with
\[
h_{\eps, m}(T)=\inf_{(x, t)\in I\times [0, T]} \la h_\eps(x, t)\ra.
\]
Moreover, \eqref{boundD:lin} and \eqref{identity:w} hold with $g=\sqrt{h_\eps^2+\eps^2}$.

Using the equation for $h_{\eps}$ and \eqref{bound:app1} we get
\bq\label{uni:dthe}
\Vert \p_th_\eps\Vert_{L^2([0, T]; H^1)}\le \cF\big(\frac{1}{h_{\eps, m}(T)+\eps},  \Vert h_0\Vert_{H^3}\big)
\eq
for all $T\le 1$. This implies
\bq\label{lower:d}
\begin{aligned}
h_\eps(x, t)&\ge h_\eps(x, 0)-\vert \int_0^t\p_t h_\eps(x, s)ds\vert \\
&\ge h_{0, m}-\sqrt{T}\Vert \p_th_\eps\Vert_{L^2([0, T]; L^\infty)}\\
&\ge h_{0, m}-\sqrt T\cF\big(\frac{1}{h_{\eps, m}(T)+\eps},  \Vert h_0\Vert_{H^3}\big)\quad\forall t\le T\le 1.
\end{aligned}
\eq
{\bf Step 2.} (Bootstrap) Denote 
\[
d_\eps(T)=\frac{1}{h_{\eps,m}(T)+\eps},\quad T\le 1.
\]
We choose $C_0$ sufficiently large and $T_0$ sufficiently small so that
\begin{align}
\label{C0:1}
C_0>\frac{1}{h_{0, m}},\\
\label{C0:2}
\sqrt{T_0}\cF(C_0, \Vert h_0\Vert_{H^3})\le \frac{h_{0, m}}{2},\\
\label{C0:3}
C_0>\frac{1}{h_{0, m}-\sqrt{T_0}\cF_2(C_0,  \Vert h_0\Vert_{H^3})}.
\end{align}
This is possible by taking 
\[
C_0>\frac{2}{h_{0, m}},\quad\sqrt{T_0}\cF_2(C_0,  \Vert h_0\Vert_{H^3})\le \frac{h_{0, m}}{2}.
\]
We claim that 
\bq\label{uni:d}
d_\eps(T_0)\le C_0\quad\forall \eps>0.
\eq
Indeed, if \eqref{uni:d} is not true then there exists $\eps_0>0$ such that $d_{\eps_0}(T_0)>C_0$. By \eqref{C0:1},
\[
d_{\eps_0}(0)=\frac{1}{h_{0,m }+\eps}\le \frac{1}{h_{0,m}}<C_0.
\]
By the continuity of $d_{\eps_0}(\cdot)$, there exists $T_1\in (0, T_0)$ such that $d_{\eps_0}(T_1)=C_0$. Then \eqref{C0:2} implies 
\[
\sqrt{T_1}\cF(d_{\eps_0}(T_1), \Vert h_0\Vert_{H^3})=\sqrt{T_1}\cF(C_0, \Vert h_0\Vert_{H^3})\le \sqrt{T_0}\cF(C_0, \Vert h_0\Vert_{H^3})\le \frac{h_{0,m}}{2}.
\]
We deduce from \eqref{lower:d} that 
\[
\inf_{I\times [0, T_1]}h_{\eps_0}\ge \mez h_{0, m}>0
\]
and
\[
h_{\eps_0, m}(T_1)\ge h_{0, m}-\sqrt{T_0}\cF(C_0, \Vert h_0\Vert_{H^3})>0.
\]
Hence
\[
C_0=d_{\eps_0}(T_1)=\frac{1}{h_{\eps_0, m}(T_1)+\eps_0}\le \frac{1}{h_{0, m}-\cF(C_0, \Vert h_0\Vert_{H^3})}.
\]
This contradicts \eqref{C0:3}, and thus we conclude the claim \eqref{uni:d}. Coming back to \eqref{lower:d} we find 
\[
\inf_{I\times [0, T_0]}h_\eps \ge \mez h_{0, m}\quad\forall \eps>0.
\]

{\bf Step 3.} (Conclusion of the argument)
Inserting \eqref{uni:d} into \eqref{bound:app1} and \eqref{uni:dthe} yields
\[
\Vert h_\eps\Vert_{X(T_0)}+\Vert\p_th_\eps\Vert_{L^2([0, T_0]; H^1(I))}\le M_0
\]
for some $M_0$ depending only on $\Vert h_0\Vert_{H^3(I)}$ and $h_{0, m}$. Set $\eps=\frac{1}{n}$ and rename $h_n=h_{\eps}$, $d_n=d_\eps$. According to Aubin-Lions's lemma, there exists $h\in X(T_0)$ such that
\begin{align}
\label{conv1}
& h_n\rightharpoonup h\quad\text{in}~ L^2([0, T_0]; H^5(I)),\\
\label{conv2}
&h^n\rightharpoonup *~h\quad\text{in}~L^\infty([0, T_0]; H^3(I)),\\
\label{conv3}
& h_n\to h\quad\text{in}~ C([0, T_0]; C^2(\overline{I})).
\end{align}
Moreover, it is easy to check that $h$ solves the problem \eqref{dp0}-\eqref{bc}.
Letting $\eps\to 0$ in \eqref{lower:d} we find
\[
\inf_{I\times [0, T_0]}h\ge \mez h_{0, m}>0.
\]
Next, it follows from \eqref{bound:app1}  and the convergences \eqref{conv1}, \eqref{conv2} that
 \[
\Vert h\Vert_{X(T_0)}\le \liminf_{n\to \infty}\Vert h_n\Vert_{X(T_0)}\le \liminf_{n\to \infty} {\cF\big(\frac{1}{h_{n, m}(T_0)+\frac 1n}, \Vert h_0\Vert_{H^3}\big).}
\]
We can replace $\liminf$ by $\lim$ of a subsequence $n_k\to \infty$. For some $(x_k, t_k)\in I\times [0, T_0]$, $h_{n_k, m}(T_0)= h_{n_k}(x_k, t_k)$.
By the compactness of $[-1, 1]\times [0, T_0]$, there exists a subsequence  $n_{k_j}\to\infty$ such that
\[
(x_{k_j}, t_{k_j})\to (x_0, t_0)\in [-1, 1]\times [0, T_0],\quad h_{n_{k_j}}(x_{k_j}, t_{k_j})\to h(x_0, t_0)\ge \inf_{I\times [0, T_0]}h
\]
where \eqref{conv3} was used in the second convergence. Consequently
\begin{align*}
\Vert h\Vert_{X(T_0)}&\le  {\cF\big(\frac{1}{\lim_{j\to \infty}h_{n_{k_j}}(x_{k_j}, t_{k_j})+\frac{1}{n_{k_j}}}, \Vert h_0\Vert_{H^3}\big)}\\
&\le  {\cF\big(\frac{1}{\inf_{I\times [0, T_0]}h}, \Vert h_0\Vert_{H^3}\big)}
\end{align*}
where the fact that $\cF$ is increasing was used. 

In addition, passing to the limit in \eqref{boundD:lin} and \eqref{identity:w} leads to \eqref{D:L1}  and \eqref{identity:D} repsectively.

Finally, because $h$ is positive on $I$, it is unique by the same argument as in section \ref{section:unique:nd}.
\section{Proof of Proposition \ref{prop:leaste}}
 Let $h\in H^1(I)$ be a nonnegative function satisfying $h(\pm 1)=1$. We have
\begin{align*}
E(h(t))=&\mez\int_I|\partial_x h|^2dx+P\int_I hdx\\
&=\mez\int_I|\partial_x(h-h_P)|^2dx+\mez\int_I |\partial_xh_P|^2dx+\int_I \partial_x(h-h_P)\partial_xh_Pdx+P\int_I hdx.
\end{align*}
Integration by parts in the cross term gives
\[
\int_I \partial_x(h-h_P)\partial_xh_Pdx=(h-h_P)\partial_xh_P\vert_{-1}^1-\int_I (h-h_P)\partial_x^2h_Pdx=-\int_I (h-h_P)\partial_x^2h_Pdx
\]
since $h=h_P$ at $\pm 1$.

\underline{Case 1:} $P\in (0, 2]$. In this case $\partial_x^2h_P=P$, and thus
\[
E(h(t))=\mez\int_I|\partial_x(h-h_P)|^2dx+\mez\int_I |\partial_xh_P|^2dx+P\int_Ih_P\ge E(h_P).
\]
Moreover, $E(h(t))=E(h_P)$ if and only if $\partial_x(h-h_P)=0$ which is equivalent to $h=h_P$ by the boundary condition $h(\pm 1)=h_P(\pm 1)=1$.

\underline{Case 2:} $P>2$. Then $\partial_x^2h_P(x)=P$ if $|x|>x_P$ and $=0$ if $|x|<x_P$. Thus
\begin{align*}
E(h(t))&=\mez\int_I|\partial_x(h-h_P)|^2dx+\mez\int_I |\partial_xh_P|^2dx+P\int_Ih-P\int_{x_P<|x|<1}(h-h_P)\\
&=\mez\int_I|\partial_x(h-h_P)|^2dx+\mez\int_I |\partial_xh_P|^2dx+P\int_{x_P<|x|<1}h_P+P\int_{-x_P}^{x_P}h\\
&=\mez\int_I|\partial_x(h-h_P)|^2dx+\mez\int_I |\partial_xh_P|^2dx+P\int_Ih_P+P\int_{-x_P}^{x_P}h\\
&\ge E(h_P).
\end{align*}
Moreover, $E(h(t))=E(h_P)$ if and only if 
\[
\begin{cases}
\partial_x(h-h_P)=0\quad\text{on}~ I,\\
h=0\quad\text{on}~(-x_P, x_P).
\end{cases}
\]
Again, owing to the boundary condition $h(\pm 1)=h_P(\pm 1)=1$, this is equivalent to $h(x, \cdot)=h_P(x)$  for $|x|>x_P$ and $h=0$ on $(-x_P, x_P)$. In other words, $h=h_P$.
\section{Proof of Theorem \ref{theo:relax}}
 Let $h_n$ be sequence of nonnegative $H^3(I)$ functions satisfying \eqref{bc}. Assume that $h_n$ is uniformly bounded in $H^1(I)$ and $D(h_n)\to 0$. Note that in view of  {the Gagliardo-Nirenberg inequality 
\[ \Vert f\Vert_{L^2(I)}\le C\Vert \partial_xf\Vert_{L^2}^\mez\Vert f\Vert_{L^1}^\mez+C\Vert f\Vert_{L^1(I)}, \]
the energy} $E$ defines a norm which is equivalent to the $H^1(I)$ norm. Then, by extracting a subsequence, still denoted $t_n$, we have $h_n\rightharpoonup h_\infty$ in $H^1(I)$. In particular, 
\bq\label{conv:C}
h_n\to h_\infty\quad\text{in}~C(\overline I).
\eq
Observe that if at some $x_0\in \overline I=[-1, 1]$, $h_\infty(x_0)>0$ then for some $\delta>0$, $h_\infty \ge \frac{2}{3}h_\infty(x_0)$ on $I_{x_0, \delta}:=(x_0-\delta, x_0+\delta)\cap I$. By \eqref{conv:C}, $h_n\ge \mez h_\infty(x_0)$ on $I_{x_0,\delta}$ for sufficiently large $n$. By the definition of $D(h)$ we get
\bq\label{H3:loc}
\int_{I_{x_0,\delta}}|\partial_x^3h_n(x)|^2dx\to 0.
\eq
By interpolation, the quantity 
\[
N_3(u):=\int_{I_{x_0,\delta}}(|u|^2+|\partial_x^3u|^2)dx
\]
defines a norm which is equivalent to the $H^3(I_{x_0,\delta})$ norm. It follows from \eqref{conv:C} and \eqref{H3:loc} that  $h_n\rightharpoonup h_\infty$ in $N_3$ and
\begin{align*}
N_3(h_\infty)\le \liminf_{n\to\infty}N_3(h(t_n))&=\lim_{n\to \infty}\int_{I_{x_0,\delta}}|h_n(x)|^2dx+\lim_{n\to \infty}\int_{I_{x_0,\delta}}|\partial_x^3h_n(x)|^2dx\\
&=\int_{I_{x_0,\delta}}|h_\infty(x)|^2dx,
\end{align*}
hence
\[
\int_{I_{x_0,\delta}}|\partial_x^3h_\infty(x)|^2dx=0.
\]
We have proved that
\begin{lemm}\label{lemm1:relax}
If $h_\infty(x_0)>0$, $x_0\in \overline I$, then there exists a neighborhood $I_{x_0, \delta}=(x_0-\delta, x_0+\delta)\cap I$ of $x_0$ in which $h_n, h_\infty$ are positive, $\partial_x^3h_\infty=0$, and $h_n\to h_\infty$ in $H^3(I_{x_0, \delta})$. Consequently, $\partial_x^3 h_\infty=0$ on $Z=\{x\in  I: h_\infty(x)>0\}$, hence  $h_\infty$ is either a parabola or a straight line on each connected component (which are open intervals) of $Z$.
\end{lemm}
The next lemma rules out the possibility that $h_n$ goes down to $0$ at a non-zero angle.
\begin{lemm}\label{lemm2:relax}
Let $x_0\in I$ and $J=(x_0, x_0+\delta)\subset I$. Let $k\in C^2(J)$ be  such that $k> 0$ on $J$ and $k$, $\partial_xk$, $\partial_x^2k$ are right-continuous at $x_0$ with $k(x^+_0)=0$ and $\partial_xk(x_0^+)\ne 0$.  Let $k_n$ be a sequence of nonnegative functions in  $H^3(I)$ such that $k_n(\pm 1)=c>0$ and $k_n\rightarrow k$ in $C^2(J)$. Then, 
\[
\int_I k_n |\partial_x^3k_n|^2\not\rightarrow 0.
\]
The same conclusion holds if $J$ is placed by $(x_0-\delta, x_0)\subset I$ and $x_0^+$ is replaced by $x_0^{-}$ in the assumptions on $k$.
\end{lemm}
\begin{proof}
Assume by contradiction 
\bq\label{ruleout:0}
\int_I k_n |\partial_x^3k_n|^2\to 0.
\eq
 Then in view of H\"oder's inequality and the boundedness of $k_n$ in $L^\infty(I)$, we have for any $I'\subset I$ that
\[
\la \int_{I'} k_n \partial_x^3k_n\ra\le \sqrt{|I'|}\Big(\int_{I'} k^2_n |\partial_x^3k_n|^2dx\Big)^\mez\le \sqrt{|I'|}\sup_n \Vert k_n\Vert_{L^\infty(I)}\Big(\int_{I'} k_n |\partial_x^3k_n|^2dx\Big)^\mez
\]
from which it follows that 
\bq\label{ruleout:1}
\int_{I'} k_n \partial_x^3k_n\to 0\quad\forall I'\subset I.
\eq
Since
\[
k(x_0^+)\partial_x^2k(x_0^+)-\frac{1}{2}(\partial_xk(x_0^+))^2=-\frac{1}{2}(\partial_xk(x_0^+))^2<0
\]
there exists $\eps\in (0, \delta)$ so small that $\partial_xk(x_0+\eps)\ne 0$ and
\[
k(x_0+\eps)\partial_x^2k(x_0+\eps)-\frac{1}{2}(\partial_xk(x_0+\eps))^2<0.
\]
Here, the assumptions that $k\in C^2(J)$ and $k, \partial_xk, \partial_x^2k$ are right continuous at $x_0$ were used. We note that $k_n(x)\ge c>0$ on $J_1=(x_0+\eps, x_0+\delta)$ for all $n$. This combined with \eqref{ruleout:0} yields $\int_{J_1}|\partial_x^3k_n|^2\to 0$, and thus $k_n\to k$ in $H^3(J_1)$ since we know $k_n\to k$ in $C^0(J_1)$. In particular, $k\in C^2(\overline{J}_1)$ and 
\[
k_n(x_0+\eps)\to k(x_0+\eps)>0,\quad \partial_xk_n(x_0+\eps)\to \partial_xk(x_0+\eps)\ne 0,\quad \partial_x^2k_n(x_0+\eps)\to \partial_x^2k(x_0+\eps).
\]
 Let $x_n$ be the global minimum of $k_n$ on $\overline I$. We know that $k_n\ge 0$, $k_n(\pm 1)=c>0$ and $k_n(x_0)\to k(x_0)=0$, hence $x_n\in I$ for $n$ sufficiently large. Then $\partial_xk(x_n)=0$ and $\partial_x^2k_n(x_n)>0$. Now we compute
\begin{align*}
\int_{x_n}^{x_0+\eps} k_n \partial_x^3k_n&=k_n\partial_x^2k_n\Big|_{x_n}^{x_0+\eps}-\int_{x_n}^{x_0+\eps} \partial_xk_n\partial_x^2k_n\\
&=k_n(x_0+\eps)\partial_x^2k_n(x_0+\eps)-k_n(x_n)\partial_x^2k_n(x_n)-\frac{1}{2}(\partial_xk_n(x_0+\epsilon))^2+\frac{1}{2}(\partial_xk_n(x_n))^2\\
&=k_n(x_0+\eps)\partial_x^2k_n(x_0+\eps)-k_n(x_n)\partial_x^2k_n(x_n)-\frac{1}{2}(\partial_xk_n(x_0+\eps))^2.
\end{align*}
Since $k_n(x_n)\partial_x^2k_n(x_n)\ge 0$,  the right-hand side is smaller than or equal to 
\[
k_n(x_0+\eps)\partial_x^2k_n(x_0+\eps)-\frac{1}{2}(\partial_xk_n(x_0+\eps))^2
\]
which converges to 
\[
k(x_0+\eps)\partial_x^2k(x_0+\eps)-\frac{1}{2}(\partial_xk(x_0+\eps))^2<0
\]
while the left-hand side converges to $0$, according to \eqref{ruleout:1}. This contradiction concludes the proof. 
\end{proof}
\huy{}We now proceed to show $h_\infty=h_P$. First, $h_\infty(1)=\lim h_n(1)=1$.  By Lemma \ref{lemm1:relax}, there exists $\delta_0\in (0, 1)$ such that $h_n\to h_\infty$ in $H^3((1-\delta_0, 1))$, $h_\infty>0$ and $\partial_x^3h_\infty=0$ on $(1-\delta_0, 1)$. In particular,  $h_n\to h_\infty$ in $C^2([1-\delta_0, 1])$ and $\partial_x^2h_\infty(1)=\lim\partial_x^2h_n(1)=P$.  Let $J=(1, 1-\delta)$ be the connected component of $Z=\{x\in I: h_\infty(x)>0\}$ whose closure contains $1$. Then $h_\infty$ is a parabola of the form 
\bq\label{formh}
h_\infty(x)=\frac{P}{2}x^2+ax+b,\quad \frac{P}{2}+a+b=1
\eq
on $J$.

\underline{Case 1}: $P\in (0, 2)$. We claim that $\delta>1$. Assume by contradiction $\delta\le1$. Then $h_\infty(x_0)=0$ with $x_0:=1-\delta\in [0, 1)$. According to Lemma \ref{lemm2:relax}, $\partial_xh_\infty(x_0)=0$. This is equivalent to
\[
\begin{cases}
\Delta:=a^2-2P(1-a-\frac{P}{2})=(a+P)^2-2P=0,\\
x_0=-\frac{a}{P},
\end{cases}
\]
where the first condition is equivalent to $a=a_1=\sqrt{2P}-P$ or $a=a_2=-\sqrt{2P}-P$. If $a=a_1$ then  $x_0=-\frac{\sqrt{2P}-P}{P}=1-\sqrt\frac{2}{P}<0$. If $a=a_2$ then $x_0=\frac{\sqrt{2P}+P}{P}>1$.  Both cases being impossible, we conclude that $\delta>1$. In particular, $h$ assumes the form \eqref{formh} on $[-\eps, 1]$ with some $\eps>0$. 

Similarly, if we start from $x=-1$ we also have that $h_\infty(x)=\frac{P}{2}x^2+a'x+b'$ for $x\in [-1, \eps']$ for some $\eps'\in (0, 1)$ and $a', b'\in \Rr$. Necessarily $ax+b=a'x+b'$ on $[-\eps, \eps']$, and thus  $(a', b')=(a, b)$. In other words, $h_\infty$ assumes the form \eqref{formh} on the whole interval $[-1, 1]$. Equalizing $h_\infty(-1)=h_\infty(1)=1$ leads to $a=0$. We thus conclude that 
\[
h(x)=\frac{P}{2}(x^2-1)+1=h_P\quad \text{on}~[-1, 1].
\]
\underline{Case 2:} $P\ge 2$. Arguing as in Case 1 we find $\delta \le 1$ and $h_\infty(x_0)=0$ with 
\[
x_0=1-\delta=1-\sqrt{\frac{2}{P}}=x_P\in [0, 1),
\]
and $a=\sqrt{2P}-P$.

When $P=2$, $x_0=0$ and $a=0$. Hence $h_\infty(x)=x^2$ on $[0, 1]$. A similar argument also gives $h_\infty(x)=x^2$ on $[-1, 0]$, hence $h_\infty=h_P$.

Consider now the case $P>2$. Then $x_0=x_P\in (0, 1)$ and 
\[
h_\infty(x)=\frac{P}{2}x^2+ax+b=\frac{P}{2}x^2+(\sqrt{2P}-P)x+1-\sqrt{2P}+\frac{P}{2}=\frac{P}{2}(x-x_P)^2
\]
on $[x_P, 1]$.  We claim that $h_\infty=0$ on $[0, x_P)$, then by symmetry $h_\infty=h_P$. Assume by contradiction $h_\infty(x_1)>0$ for some $x_1\in [0, x_P)$. Let $(a, b)\subset I$ be the connected component of $Z=\{x\in I:h_\infty>0\}$ that contains $x_1$. Necessarily $h_\infty(b)=0$ and $b\le x_P$. By Lemma \ref{lemm1:relax}, $h_\infty$ is either a parabola or a straight line  $(a, b)$. Let us show that both cases are impossible. Indeed, if $h_\infty$ is a straight line on $(a, b)$ then $h_\infty$ hits $0$ at $x=b$ (from the left) with an angle, which is impossible according to Lemma \ref{lemm2:relax}. Assume now that $h_\infty$ is a parabola on $(a, b)$. Since $h_\infty$ must touch down from the left of $b$ at zero angle, the only possibility is that the parabola  $\frac{P}{2}x^2+ax+b$ is positive while its slope is negative on $(-\infty, b)$. Thus $h_\infty(x)=\frac{P}{2}x^2+ax+b$ on the whole interval $[-1, b]$. But then $h_\infty(-1)=h_\infty(1)=1$ yields $a=0$ which contradicts the fact that $a=\sqrt{2P}-P<0$. Therefore, $h_\infty=h_P$ when $P>2$. \huy{}

By Lemma \ref{lemm1:relax}, $h_n\to h_P$ in $H^3_{loc}(\{x: h_P(x)>0\})$ for any $P>0$. Furthermore, when $P\in (0, 2)$, $h_P>0$ on $I$ and one can take in Lemma \ref{lemm1:relax} $I_{x_0, \delta}=I$ for any $x_0\in I$, hence $h_n\to h_P$ in $H^3(I)$. We have actually proved that any subsequence of $(h_n)$ has a subsequence with desired convergence properties. Because the limit is unique (and is equal to $h_P)$ we conclude that in fact the whole sequence $h_n$ has those properties.
\section{Proof of Theorem \ref{coro:stability}}
{\bf Part 1.} Let $P\in (0, 2)$, and let $h_0\in H^3(I)$ satisfy \eqref{bc} and $\inf_I h_0>0$. According to Theorem \ref{Cauchy}, there exist a maximal time of existence $T^*\in (0, \infty]$ and a unique solution $h\in X(T)$ with $\inf_{I\times [0, T]}h>0$ for any $T<T^*$. Set $u=h-h_P$, then because $\partial_x^3h_P=0$ we have
\bq\label{pde:sta}
\begin{cases}
  \partial_tu(x, t)+\partial_x(h\partial_x^3u)(x, t)=0, &\quad(x, t)\in I\times (0, T^*),\\
 u(\pm 1, t)=\partial_x^2u(\pm 1, t)=0,&\quad t>0.
 \end{cases}
 \eq
Multiplying the first equation in \eqref{pde:sta} by $-\partial_x^2u$ and integrating by parts, we obtain as in section \ref{section:H1},
\bq\label{stab:dt}
\mez\frac{d}{dt}\Vert \partial_x u(\cdot, t)\Vert^2_{L^2(I)}=-\int_I h(t, x)|\partial_x^3u(x, t)|^2dx,\quad t\in (0, T^*).
\eq
In particular,
\[
\Vert \partial_x u(\cdot, t)\Vert_{L^2(I)}\le \Vert \partial_x u(\cdot, 0)\Vert_{L^2(I)},\quad t\in (0, T^*).
\]
Since $u(\pm 1, \cdot)=0$, Poincar\'e's inequality together with the embedding $H^1(I)\subset C(I)$ yields
\[
\Vert u(\cdot, t)\Vert_{L^\infty(I)}\le C_1\Vert \partial_x u(\cdot, t)\Vert_{L^2(I)}\le C_1\Vert \partial_x u(\cdot, 0)\Vert_{L^2(I)},\quad t\in (0, T^*).
\]
Consequently,
\[
h(x, t)\ge h_\infty(x)-C_1\Vert \partial_x u(\cdot, 0)\Vert_{L^2(I)}\ge \frac{2-P}{2}-C_1\Vert \partial_x u(\cdot, 0)\Vert_{L^2(I)},
\]
and thus
\bq\label{stab:pos}
h(x, t)\ge \mez \frac{2-P}{2}
\eq
for all $(x, t)\in I\times [0, T^*)$ provided 
\[
\Vert \partial_x u(\cdot, 0)\Vert_{L^2(I)}\le \frac{1}{2C_1}\frac{2-P}{2}.
\]
Therefore, $T^*=\infty$ according to the blow-up criterion \eqref{blowup}. 

Next, we show that $h$ converges to $h_\infty$ exponentially in $H^1(I)$. Indeed, because $\partial_x^2u(\pm 1, \cdot)=0$ and $\int_I \partial_xu dx=u(1)-u(-1)=0$, Poincar\'e's inequalities yield
\[
\Vert \partial_x^3u(x, t)\Vert_{L^2(I)}\ge C_2\Vert \partial_x^2u(x, t)\Vert_{L^2(I)}\ge C_3\Vert \partial_xu(x, t)\Vert_{L^2(I)}
\]
which combined with \eqref{stab:pos} and \eqref{stab:dt} leads to
\[
\frac{d}{dt}\Vert \partial_x u(\cdot, t)\Vert^2_{L^2(I)}\le -C_4\Vert \partial_xu(\cdot, t)\Vert_{L^2(I)}^2.
\]
By Gr\"onwall's lemma,
\[
\Vert \partial_x u(\cdot, t)\Vert_{L^2(I)}\le \Vert \partial_x u(\cdot, 0)\Vert_{L^2(I)}\exp(-C_4t)\quad\forall t>0.
\]
Finally, note that $u(\pm 1, \cdot)=0$ we conclude by Poincar\'e's inequality that
\bq\label{stab}
\Vert u(\cdot, t)\Vert_{H^1(I)}\le C\Vert  u(\cdot, 0)\Vert_{H^1(I)}\exp(-{C_4} t)\quad\forall t>0.
\eq
Let us now turn to prove that $D(h)\in W^{1, 1}(\Rr^+)$. According to \eqref{D:L1}, $D(h)\in L^1(\Rr^+)$. Thus, by virtue of \eqref{identity:D}, it remains to show that
\[
A:=\int_I \p_th|\p_x^3h|^2(x, s)dx-2\int_I |\p_x\p_th|^2(x, s)dx\in L^1(\Rr^+).
\]
In the rest of this proof, we write $L^pL^q\equiv L^p(\Rr^+; L^q(I))$. We first note that by \eqref{dt:g},
\bq\label{conv:dth:L2}
 \Vert \p_th\Vert_{L^2L^2}\le C\Vert h\Vert_{L^\infty H^2}(\Vert \p_x^3h\Vert_{L^2L^2}+\Vert \p_x^4h\Vert_{L^2 L^2}).
\eq
Consider next $\p_x\p_th=-\p_x^2h\p_x^3h-2\p_xh\p_x^4h-h\p_x^5h$.
It is readily seen that
\[
\Vert\p_xh\p_x^4h\Vert_{L^2L^2}\le C\Vert h\Vert_{L^\infty H^2}\Vert \p_x^4h\Vert_{L^2 L^2},\quad \Vert h\p_x^5h\Vert_{L^2L^2}\le C\Vert h\Vert_{L^\infty H^1}\Vert \p_x^5h\Vert_{L^2 L^2}.
\]
Using \eqref{d3d4} we bound
\begin{align*}
\Vert \p_x^2h\p_x^3h\Vert_{L^2L^2}\le\Vert \p_x^2h\Vert_{L^\infty L^2} \Vert \p_x^3h\Vert_{L^2L^\infty}\le C\Vert \p_x^2h\Vert_{L^\infty L^2}\Vert \p_x^4h\Vert_{L^2L^2}.
\end{align*}
Consequently
\bq\label{conv:dth:H1}
\Vert \p_x\p_th\Vert_{L^2L^2}\le {C\Vert h\Vert_{L^\infty H^2}\Vert \p_x^4h\Vert_{L^2 L^2}+C\Vert h\Vert_{L^\infty H^1}\Vert \p_x^5h\Vert_{L^2 L^2}}.
\eq
In view of the lower bound \eqref{stab:pos}, it follows from \eqref{Cauchy:apriori} that
\bq\label{conv:global}
\Vert h\Vert_{X(\Rr^+)}\le \cF(\Vert h_0\Vert_{H^3}\big).
\eq
This together with  \eqref{conv:dth:H1} yields 
\bq\label{conv:L11}
\int_0^\infty \int_I |\p_x\p_th|^2(x, s)dxds=\Vert \p_x\p_th\Vert_{L^2L^2}^2\le  \cF(\Vert h_0\Vert_{H^3}\big).
\eq
On the other hand, using \eqref{d3d4} and H\"older's inequality we get
\[
\int_I\p_th|\p_x^3h|^2dx\le \Vert \p_th\Vert_{L^2(I)}\Vert \p_x^3h\Vert_{L^2(I)}\Vert \p_x^3h\Vert_{L^\infty(I)}\le C\Vert \p_th\Vert_{L^2(I)}\Vert \p_x^3h\Vert_{L^2(I)}\Vert \p_x^4h\Vert_{L^2(I)},
\]
hence
\[
\int_0^\infty \la  \int_I \p_th|\p_x^3h|^2(x, s)dx\ra ds\le C\Vert \p_th\Vert_{L^2L^2}\Vert \p_x^3h\Vert_{L^\infty L^2}\Vert \p_x^4h\Vert_{L^2 L^2}.
\]
Employing \eqref{conv:dth:L2} and \eqref{conv:global} we deduce that
\[
\int_0^\infty \la \int_I \p_th|\p_x^3h|^2(x, s)dx\ra ds\le  \cF(\Vert h_0\Vert_{H^3}\big)
\]
which combined with \eqref{conv:L11} concludes that $A\in L^1(\Rr^+)$. This completes the proof of $D(h)\in W^{1,1}(\Rr^+)$. According to Corollary 8.9 \cite{Brezis} we then have $D(h(t))\to 0$ as $t\to \infty$, and thus Theorem \ref{theo:relax} implies that $h(t)\to h_P$ in $H^3(I)$ as $t\to \infty$.

{\bf Part 2.} Let $P\ge 2$, and let $h_0\in H^3(I)$ satisfy \eqref{bc} and $\inf_I h_0>0$. Suppose that the solution $h$ to \eqref{dp0}-\eqref{bc} with initial data $h_0$ is not pinched at finite time neither at infinite time, then according to Theorem \ref{Cauchy}, $h$ is global, $h\in X(T)$ for any $T>0$, and 
\bq\label{pinch:pos}
\inf_{I\times [0, \infty)}h\ge c_0
\eq
 for some $c_0>0$. Set
\[
h_\infty(x)=\frac{P}{2}(x^2-1)+1.
\]
Observe that $h_\infty$ is a stationary solution of \eqref{dp0}-\eqref{bc} and $h_\infty$ vanishes at $\pm\sqrt{\frac{P}{2}-1}$.
As before, $u=h-h_\infty$ satisfies \eqref{pde:sta}. By virtue of \eqref{pinch:pos}, the proof of \eqref{stab} also gives
 \[
\Vert u(\cdot, t)\Vert_{H^1(I)}\le C\Vert  u(\cdot, 0)\Vert_{H^1(I)}\exp(-Ct)\quad\forall t>0.
\]
In particular,
\[
\lim_{t\to\infty}\Vert h(\cdot, t)-h_\infty(\cdot)\Vert_{C(I)}=0.
\]
Because $h_\infty(\sqrt{\frac{P}{2}-1})=0$, we deduce that $\lim_{t\to \infty} h(\sqrt{\frac{P}{2}-1}, t)=0$ which contradicts \eqref{pinch:pos}.

Assume now that $h$ is global in time. Since $D(h)\in L^1(\Rr^+)$ there exists $t_n\to \infty$ such that $D(h(t_n))\to 0$. By virtue of Theorem \ref{theo:relax},  $h(t_n)\rightharpoonup h_P$ in $H^1(I)$ and $h(t_n)\to h_P$ in $H^3_{loc}(\{x:h_P(x)>0\})$.

\appendix
\section{Weak solutions}
\label{sec:appendix}
\begin{theo}[\bf Existence of global weak solutions]
\label{theo:weak}
Let $h_0\in H^1(I)$ be a nonnegative function such that $h_0\in H^3((-1, -1+\delta_0))\cap H^3((1-\delta_0, 1))$ for some $\delta_0\in (0, 1)$ and $h_0$ satisfies \eqref{bc}. Let $T$ be a positive real number. Then there exists a global weak solution $h$ of \eqref{dp0}-\eqref{bc} in the sense of Definition \ref{weak}. More precisely,
\[
h\in C(\overline I\times [0, T])\cap L^\infty([0, T]; H^1(I))\cap L^2([0, T]); H^2(I))\cap H^1((0, T); H^{-1}(I))
\]
and there exists $\delta\in (0, 1)$ independent of  $T$ such that 
\[
h\in L^2\big([0, T]; H^3((-1, -1+\delta))\cap H^3((1-\delta, 1))\big).
\]
\end{theo}
\begin{proof}
Let $h_0^n\in H^3(I)$ be a sequence of nonnegative functions satisfying \eqref{bc} such that  $h_0^n\to h$ in $H^1(I)\cap H^3(J)$. According to Theorem \ref{Cauchy:nd} there exists for each $n$ a unique solution $h^n\in X([0, T])$, for any $T>0$, to the problem
\bq\label{dp:nd1}
\begin{cases}
 \partial_th^n(x, t)+\partial_x(\sqrt{ |h^n|^2+n^{-2}}\partial_x^3h^n)(x, t)=0, &\quad(x, t)\in I\times (0, \infty),\\
h^n(\pm 1, t)=1, \partial_x^2h^n(\pm 1, t)=P, &\quad t>0,\\
h^n(x, t)=h^n_0(x), &\quad t=0.
 \end{cases}
 \eq
 Moreover, there exists $C>0$ independent of $n$ and $T$ such that
 \bq\label{weak:bound1}
 \Vert h^n\Vert_{L^\infty([0, T]; H^1(I))}\le C\Vert h^n_0\Vert_{H^1(I)}
 \eq
 and 
 \bq\label{weak:bound2}
\int_0^T\int_I g^n|\p_x^3h^n|^2(x, s)dxds\le C(\Vert h^n_0\Vert_{H^1(I)}^2+1),\quad g^n=\sqrt{|h^n|^2+n^{-2}}.
 \eq
 Writing $g^n\p_x^3h^n=\p_x(g^n\p_x^2h^n)-\p_xg^n\p_x^2h^n$ we have
 \[
0= \partial_th^n+\partial_x(g^n\partial_x^3h^n)=\partial_th^n+\partial_x^2(g^n\p_x^2h^n)-\p_x(\p_xg^n\p_x^2h^n).
 \]
 Then, for any $\varphi\in C_0^\infty(I\times (0, T))$,
 \bq\label{weak:n}
 -\int_0^T\int_Ih^n\p_t\varphi+\int_0^T\int_Ig^n\p_x^2h^n\p_x^2\varphi+\int_0^T\int_I \p_xg^n\p_x^2h^n\p_x\varphi=0.
 \eq
 Because $h^n(\pm 1, \cdot)=1$ and $h^n$ is uniformly bounded in $L^\infty(\Rr^+; C^\mez(\overline I))$ (by virtue of \eqref{weak:bound1} and the embedding $H^1(I)\subset C^\mez(\overline I)$), there exists $\delta>0$ sufficiently small such that 
 \[
 h^n(x, t)\ge \mez\quad\forall t\ge 0,~\forall x \in J_1:=[-1, -1+\delta]\cup [1-\delta, 1]:=J_{1,l}\cup J_{1, r}.
 \]
 It then follows from \eqref{weak:bound2} that
 \bq\label{weak:bound4}
\Vert \p_x^3h^n\Vert_{L^2(\Rr^+; L^2(J_1))}\le C=C(\Vert h_0\Vert_{H^1(I)})
 \eq
which  combined with \eqref{weak:bound1} and interpolation yields
\bq\label{weak:b5}
\Vert h^n\Vert_{L^2([0, T]; H^3(J_1))}\le C=C(\Vert h_0\Vert_{H^1(I)}, T),\quad\forall T>0.
\eq
 Let $A>0$ depend only on $\Vert h_0\Vert_{H^1(I)}$ such that
 $\Vert h^n\Vert_{L^\infty(I\times \Rr^+)}\le A$ for all $n$. We define
 \[
 f_n(s)=-\int_s^A \frac{dr}{\sqrt{r^2+n^{-2}}},\quad F_n(s)=-\int_s^A f_n(r)dr.
 \]
 Note that $g_n(s)\le 0$ and $F_n(s)\ge 0$ for any $s\le A$. Let $\chi$ be a nonnegative cut-off function equal to $1$ on $I_1:=I\setminus J_1$ and supported on $(-1, 1)$. Multiplying the first equation in \eqref{dp:nd1} by $f_n(h^n(x, t))\chi(x)$ then integrating by parts we obtain
 \begin{align*}
 \int_I \partial_th^nf_n(h^n)\chi dx&=-\int_I\partial_x(g^n\partial_x^3h^n)f_n(h^n)\chi dx\\
 &=\int_I g^n\partial_x^3h^nf_n'(h^n)\p_xh^n\chi dx+\int_I g^n\partial_x^3h^n f_n(h^n)\p_x\chi dx\\
 &=\int_I \partial_x^3h^n\p_xh^n\chi dx+\int_Ig^n\partial_x^3h^n f_n(h^n)\p_x\chi dx\\
 &=-\int_I |\partial_x^2h^n|^2\chi dx-\int_I \p_x^2h^n\p_xh^n\p_x\chi dx+\int_Ig^n\partial_x^3h^n f_n(h^n)\p_x\chi dx.
 \end{align*}
 Since
\[
\int_I \partial_th^nf_n(h^n)\chi dx=\frac{d}{dt}\int_I F_n(h^n)\chi dx
\] 
we deduce that
\bq\label{weak:ineq}
\begin{aligned}
&\int_I F_n(h^n)(x, T)\chi dx+\int_0^T\int_{I_1} |\partial_x^2h^n|^2\chi dxds\\
&\le \int_I F_n(h^n)(x, 0)\chi dx-\int_0^T\int_I \p_x^2h^n\p_xh^n\p_x\chi dxds+\int_0^T\int_Ig^n\partial_x^3h^n f_n(h^n)\p_x\chi dxds.
\end{aligned}
\eq
We split
\[
\int_0^T\int_I \p_x^2h^n\p_xh^n\p_x\chi dx=\int_0^T\int_{I_1} \p_x^2h^n\p_xh^n\p_x\chi dx+\int_0^T\int_{J_1}\p_x^2h^n\p_xh^n\p_x\chi dx=:H_1+H_2.
\]
Using H\"older's inequality and \eqref{weak:bound1} we get
\[
|H_1|\le C\Vert \p_x^2h^n\Vert_{L^2([0, T]; L^2(I_1))},\quad  C=C(\Vert h_0\Vert_{H^1(I)}, T).
\]
On the other hand, \eqref{weak:b5} gives
\[
|H_2|\le C=C(\Vert h_0\Vert_{H^1(I)}, T).
\]
Thus
\bq\label{weak:b6}
\la \int_0^T\int_I \p_x^2h^n\p_xh^n\p_x\chi dxds\ra \le C\Vert \p_x^2h^n\Vert_{L^2([0, T]; L^2(I_1))}+C,\quad C=C(\Vert h_0\Vert_{H^1(I)}, T).
\eq
Applying H\"older's inequality together with \eqref{weak:bound1} and \eqref{weak:bound2} we find
\bq\label{weak:b7}
\la \int_0^T\int_Ig^n\partial_x^3h^n f_n(h^n)\p_x\chi dxds\ra\le C=C(\Vert h_0\Vert_{H^1(I)}, T).
\eq
In addition, it is easy to see that 
\bq\label{weak:b8}
\int_I F_n(h^n)(x, 0)\chi dx\le C=C(\Vert h_0\Vert_{H^1(I)}).
\eq
Putting together \eqref{weak:ineq}, \eqref{weak:b6},  \eqref{weak:b7} and  \eqref{weak:b8} yields
\[
\Vert \p_x^2h^n\Vert^2_{L^2([0, T]; L^2(I_1))}\le C\Vert \p_x^2h^n\Vert_{L^2([0, T]; L^2(I_1))}+C,\quad C=C(\Vert h_0\Vert_{H^1(I)}, T).
\]
Consequently, there exists $C=C(\Vert h_0\Vert_{H^1(I)}, T)$ such that 
\[
\Vert \p_x^2h^n\Vert_{L^2([0, T]; L^2(I_1))}\le C\quad\forall n.
\]
 This together with \eqref{weak:b5} implies
\bq\label{weak:b9}
\Vert \p_x^2h^n\Vert_{L^2([0, T]; L^2(I))}\le C\quad\forall n.
\eq
Let us fix a positive (finite) time $T$. A combination of \eqref{weak:bound1} and \eqref{weak:bound2} leads to the uniform boundedness of $g^n\p_x^3h^n$ in $L^2([0, T]; L^2(I))$, hence the uniform boundedness of $\p_th^n$ in $L^2([0, T]; H^{-1}(I))$. Using this, \eqref{weak:bound1}, \eqref{weak:b5}, \eqref{weak:b9} and Aubin-Lions's lemma we conclude that up to extracting a subsequence,
\[
h^n\wc h~ \text{in}~L^2([0, T]; H^2(I)),\quad h^n\to h~ \text{in}~ C(\overline I\times [0, T])\cap L^2([0, T]; H^1(I))\cap L^2([0, T]; C^2(\overline{J_1}))
\]
for some 
\[
h\in C(\overline I\times [0, T])\cap L^\infty([0, T]; H^1(I))\cap L^2([0, T]; H^2(I))\cap L^2([0, T]; H^3(J_{1, l})\cap H^3(J_{1, r}))
\]
with $\p_th \in L^2((0, T); H^{-1}(I))$. In particular, $h$ satisfies the boundary conditions \eqref{bc} for {\it a.e.} $t\in [0, T]$. We claim that 
\[
h(x, t)\ge 0\quad\forall (x, t)\in I\times [0, T].
\]
Indeed, coming back to \eqref{weak:ineq} we deduce from \eqref{weak:b6}, \eqref{weak:b7} and \eqref{weak:b8} that
\bq\label{weak:b10}
\int_I F_n(h^n(x, t))dx \le C(\Vert h_0\Vert_{H^1(I)}, T)
\eq
for all $n\ge 0$ and $t\le T$.  Assume by contradiction $h(x_0, t_0)<0$ for some $(x_0, t_0)\in I\times [0, T]$.  Since $h^n\to h$ uniformly on $\overline I\times [0, T]$, there exist $\eta>0$ and $n_0\in \xN$ such that 
\[
h_n(x, t_0)<-\eta\quad\text{if}~|x-x_0|\le \delta,~n\ge n_0.
\]
But for such $x$,
\[
F_n(h^n(x, t_0))=-\int_{h^n(x, t_0)}^Af_n(s)ds\ge -\int_{-\eta}^0f_n(s)ds\to  -\int_{-\eta}^0f_\infty(s)ds\quad\text{as}~n\to \infty
\]
by the monotone convergence theorem, here 
\[
f_\infty(s)\defn\lim_{n\to \infty} f_n(s)=-\infty
\]
for any $s\le 0$. It follows that
\[
\int_I F_n(h^n(x, t_0))=+\infty
\]
which contradicts \eqref{weak:b10}, and thus $h\ge 0$.

Then letting $n\to \infty$ in \eqref{weak:n} leads to
 \bq\label{weak1}
 -\int_0^T\int_Ih\p_t\varphi +\int_0^T\int_I h\p_x^2h\p_x^2\varphi+\int_0^T\int_I \p_xh\p_x^2h\p_x\varphi=0\quad\forall \varphi\in C^\infty_0(I\times (0, T)).
 \eq
 Writing $\p_xh\p_x^2h=\mez\p_x|\p_xh|^2$ and integrating by parts in the last integral we arrive at
  \bq\label{weak2}
 -\int_0^T\int_Ih\p_t\varphi +\int_0^T\int_I \big(h\p_x^2h-\mez|\p_xh|^2\big)\p_x^2\varphi =0\quad\forall \varphi\in C^\infty_0(I\times (0, T)).
 \eq
 In other words, $h$ is a weak solution of \eqref{dp0}-\eqref{bc} in the sense of Definition \ref{weak}. 
\end{proof}
In general, weak solutions can be non-unique. Nevertheless, the steady weak solution $h_P$ is unique as shown in the next Proposition.
\begin{prop}[\bf Uniqueness of $h_P$]
\label{unique:hP}
For any $P>0$, $h_P$ is the unique even weak steady solution, in the sense of Definition \ref{weak}, to \eqref{dp0}-\eqref{bc}.\end{prop}
\begin{proof}
It is easy to check that $h_P$ is an even weak steady solution in the sense of  Definition \ref{weak}. Assume now that $h$ is an even weak steady solution, we prove that $h=h_P$. We first notice that the weak formulation \eqref{weakform} is equivalent to 
$\p_x\p_x(h\p_x^2h-\mez|\p_xh|^2)=0$ in $\D'(I)$, or again $\p_x(h\p_x^2h-\mez|\p_xh|^2)=C$ in $\D'(I)$ for some constant $C$.  We claim that $C=0$.  Indeed, writing $h\p_x^2h=\p_x(h\p_xh)-|\p_xh|^2)$ we get  
\[
C\int_I \varphi=-\langle h\p_x^2h-\mez|\p_xh|^2, \p_x\varphi \rangle_{\D'(I), \D(I)}=-(h\p_x^2h-\mez|\p_xh|^2, \p_x\varphi)_{L^2(I), L^2(I)}
\]
for any $\varphi\in \D(I)$. Noting that  $h$ is even, we can make the change of variables $x\mapsto -x$ to obtain
\[
C\int_I \varphi =\langle h\p_x^2h-\mez|\p_xh|^2, \p_x\varphi_1\rangle_{L^2(I), L^2(I)}=-C\int_I \varphi_1
\]
with $\varphi_1(\cdot)=\varphi(-\cdot)\in \D(I)$. Since $\int_I \varphi_1=\int_I \varphi$ for any $\varphi\in \D(I)$ we conclude that $C=0$ as claimed.

We thus have
\bq\label{eq:unique:hP}
\begin{aligned}
0&=(h\p_x^2h-\mez|\p_xh|^2, \p_x\varphi)_{L^2(I), L^2(I)}\\
&=(\p_x^2h, \p_x(h\varphi))_{L^2(I), L^2(I)}-(\p_x^2h, \p_xh\varphi)_{L^2(I), L^2(I)}+\mez(\p_x|\p_xh|^2, \varphi)_{L^2(I), L^2(I)}\\
&=-\langle \p_x^3h, h\varphi\rangle_{H^{-1}(I), H^1_0(I)}
\end{aligned}
\eq
for any $\varphi\in H^1_0(I)$. If $h(x_0)>0$, $x_0\in \overline I$, there exists a neighborhood $I_{x_0}$ of $x_0$ in $I$ such that $h\ge \mez h(x_0)$ on $I_{x_0}$. For any $\psi\in H^1_0(I_{x_0})$, defining 
\[
\varphi(x)=
\begin{cases}
\frac{\psi}{h},\quad x\in I_{x_0},\\
0,\quad x\in I\setminus I_{x_0}
\end{cases}
\] 
we have $\varphi\in  H^1_0(I_{x_0})\subset H^1_0(I)$ and by \eqref{eq:unique:hP},
\[
\langle \p_x^3h, \psi\rangle_{H^{-1}(I_{x_0}), H^1_0(I_{x_0})}=0.
\]
This implies $\p_x^3h=0$ in $\D'(I_{x_0})$, and thus $\p_x^3h=0$ in $\D'(\{h>0\})$. Consequently, on each connected component (which are open intervals) of $\{h>0\}$, $h$ is either a parabola or a straight line. In addition, $h$ cannot hit $0$ at a non-zero angle because $h\in H^2(I)$. We are thus in the same situation as in the proof of Theorem \ref{theo:relax} which allows us to conclude that $h=h_P$.
\end{proof}


\vspace{.2in}
\noindent{\bf{Acknowledgment.}} 
The research of PC is partially funded by NSF grant DMS-1209394. The research if VV is partially funded by NSF grant DMS-1652134 and an Alfred P.~Sloan Fellowship.


\newcommand{\etalchar}[1]{$^{#1}$}

\end{document}